\documentclass{article}
\usepackage{amsmath}
\usepackage{amsthm}
\usepackage{times}
\usepackage{bm}
\usepackage[font=small,skip=0pt]{caption}
\usepackage{subcaption}
\usepackage{graphicx}
\usepackage{amssymb}
\usepackage{pifont}
\usepackage[page,title]{appendix}
\usepackage{mathrsfs}
\usepackage{comment}
\usepackage{geometry}
\usepackage{setspace}
\usepackage{fancyhdr} 
\usepackage{listings}
\usepackage{xcolor}
\usepackage{natbib}
\usepackage{verbatim}
\usepackage{booktabs}
\usepackage{multirow}
\usepackage{enumerate}
\usepackage{appendix}
\usepackage{authblk}

\usepackage{thmtools}
\usepackage{thm-restate}

\allowdisplaybreaks[4]
\newcommand{\WFcomment}[1]{}
\newcommand{\WFsuggest}[1]{{\color{blue!80!black} {\textbf [#1]}}}
\newcommand{\simiid}{\stackrel{i.i.d.}{\sim}}
\newcommand{\simind}{\stackrel{ind}{\sim}}
\newcommand{\PP}{\mathbb{P}}
\newcommand{\EE}{\mathbb{E}}

\newtheorem{theorem}{Theorem}
\newtheorem{proposition}{Proposition}
\newtheorem{definition}{Definition}
\newtheorem{cor}{Corollary}

\newcommand{\rev}[1]{{\textcolor{black}{#1}}}
\definecolor{ForestGreen}{rgb}{0.0, 0.35, 0.13}
\newcommand{\resp}[1]{{\color{ForestGreen}{#1}}}

\begin{document}\large
\title{Optimality of the max test for detecting sparse signals with Gaussian or heavier tail}
\author{Xiao Li and William Fithian\\Department of Statistics, UC Berkeley}


\maketitle{}

\begin{abstract}
  A fundamental problem in high-dimensional testing is that of {\em global null testing}: testing whether the null holds simultaneously in all of $n$ hypotheses. The max test, which uses the smallest of the $n$ marginal p-values as its test statistic, enjoys widespread popularity for its simplicity and robustness. However, its theoretical performance relative to other tests has been called into question. In the Gaussian sequence version of the global testing problem, \citet{donoho2004higher} discovered a so-called ``weak, sparse" asymptotic regime in which the higher criticism and Berk-Jones tests achieve a better detection boundary than the max test when all of the nonzero signal strengths are identical.
  We study a more general model in which the non-null means are drawn from a generic distribution, and show that the detection boundary for the max test is optimal in the ``weak, sparse" regime, provided that the distribution's tail is no lighter than Gaussian. Further, we show theoretically and in simulation that the modified higher criticism of \citet{donoho2004higher} can have very low power when the distribution of non-null means has a polynomial tail.
\end{abstract}

\section{Introduction}

\subsection{Sparse signal detection}

Closely related to multiple testing is the problem of testing the {\em global null} or {\em intersection null}, which asserts that all of $n$ univariate null hypotheses are true; this is sometimes called the {\em signal detection problem}, since it amounts to asking whether there is any signal at all. One strategy, popular among methodologists and practitioners alike for its simplicity, transparency, and robustness, is to reject when the largest univariate test statistic is above a critical threshold, or equivalently when the smallest univariate $p$-value is below an appropriately corrected significance level. This method, called the {\em max test}, is closely associated with the multiple testing procedure that rejects individual hypotheses with $p$-values below the same threshold, which is $1-(1-\alpha)^{1/n}$ if the $p$-values are independent (called the {\em \v{S}id\'{a}k correction}), or $\alpha/n$ if the dependence structure is completely unknown (the {\em Bonferroni correction}), and may be obtained by simulation in other cases \citep{sidak1968multivariate}. Because the associated multiple testing procedure controls the familywise error rate (FWER), the max test can be tacked on as a logical deduction about the global null, incurring no additional FWER.

However, the adequacy of the max test for signal detection has been placed in doubt because it does not always achieve an optimal detection boundary in the {\em Gaussian sequence model} where we observe $X \sim N_n(\mu, I_n)$, a canonical testing ground for high-dimensional statistical methods. In certain sparse asymptotic regimes of this model, the max test is outperformed by more sophisticated special-purpose tests of the global null $H_0:\; \mu_i = 0$ for all $i$, against $H_1:\; \mu_i \neq 0$ for some $i$.


Most notably, Donoho and Jin \citeyearpar{donoho2004higher,donoho2015higher} compared the max test to the {\em higher criticism} (HC) test, which rejects the global null for large values of Tukey's higher criticism statistic $$HC_n=\sup_{1\,\leq\, i\,\leq\, n/2} \frac{\sqrt{n}(i/n-p_{(i)})}{\sqrt{p_{(i)}(1-p_{(i)})}}=\sup_{0\,\leq\, t\,\leq\, 1/2} \frac{\sqrt{n}(\widehat{F_n}(t)-t)}{\sqrt{t(1-t)}},$$
where $p_{(1)}\leq\dots\leq p_{(n)}$ are the ordered $p$-values and $\widehat{F_n}(t)$ is their empirical distribution function. They also studied two related tests: the {\em modified higher criticism } test, which rejects for large values of
\[
mHC_n= \sup_{1/n\,\leq \,t\,\leq\, 1/2} \frac{\sqrt{n}(\widehat{F_n}(t)-t)}{\sqrt{t(1-t)}},
\]
and the {\em Berk-Jones test}, which rejects for large values of
\[
BJ_n = \max_{1\,\leq\, k\, \leq\, n/2} \;(2n)^{1/2} \left\{\frac{k}{n}\,\log\left(\frac{k}{np_{(k)}}\right) +
\left(1-\frac{k}{n}\right)\,\log\left(\frac{n-k}{n(1-p_{(k)})}\right)\right\}^{1/2}.
\]
They showed, in a model where all nonzero $\mu_i$ take the same value, that the higher criticism, modified higher criticism, and Berk-Jones tests all achieve the optimal detection boundary in the sparse asymptotic regime where the number $n_1$ of nonzero signals grows more slowly than $n^{1/2}$ (for denser signals, the $\chi^2$ test is typically much more powerful than all tests under comparison here). By contrast, the max test falls short unless $n_1 = O(n^{1/4})$. In light of these results, it has been widely accepted as a stylized fact that these special-purpose tests dominate the max test for sparse signal detection.

While \cite{donoho2004higher} provide a remarkably detailed and complete picture of global testing in the asymptotic regime they study, it is natural to ask how the story might change if we relax the rather restrictive assumption that all of the nonzero signals have identical strength, since in real applications we would expect these to vary in magnitude. 
This article considers a more general setting where the non-null signals are instead drawn from a distribution $G_n$: 
\begin{equation}
\label{eq:sigmaG}
\{\mu_i\}_{i=1}^n\stackrel{i.i.d}{\sim} (1-\pi_n)\delta_0(\cdot)+\pi_nG_n(\cdot), \quad \pi_n = n^{-\beta}, \quad 0<\beta <1.
\end{equation}
This model was previously studied by \citet{cai2014optimal}, who showed under certain regularity conditions in the sparse regime $\beta > 1/2$ that the higher criticism test achieves the optimal detection boundary in the signal sparsity parameter $\beta$. In particular we will be interested in the case where all $G_n$ come from a common scale family with scale parameter $\sigma_n$. The regime of \cite{donoho2004higher} is a special case where $G_n = \delta_{\sigma_n}$ for $\sigma_n = \sqrt{2r\log n}$.

Interestingly, we find that relaxing the assumption of identical non-null signals shows the max test in a considerably better light. Our main results are summarized in the last three rows of Table~\ref{table:main}. Essentially, if the tails of $G_n$ are at least as heavy as Gaussian, the max test achieves optimal performance throughout the sparse regime, i.e. $\beta > 1/2$. Furthermore, if $G_n$ has polynomial tails, we find that the max test asymptotically dominates the modified higher criticism test; the higher criticism and Berk-Jones tests remain competitive but only because of their similarity to the max test. We give explicit formulae for the detection threshold when $G_n$ has Gaussian, exponential, and polynomial tails and confirm our results with numerical experiments. We find empirically that a hybrid test combining the max test and $\chi^2$ test is a practical choice with high power across all sparsity levels.

\begin{table*}\centering
\begin{tabular}{@{}rcccccccc@{}}\toprule
&& \multicolumn{2}{c}{Asymptotic parameters} & \phantom{abc} & \multicolumn{3}{c}{Achieves optimal asymptotic behavior}\\
\cmidrule{3-4} \cmidrule{6-8}
Alternative distribution & \phantom{abc} & $\sigma_n$ & $\beta$  
    & \phantom{abc} &  Max test & Higher criticism  & Modified HC  \\\midrule
\multirow{2}{*}{Point mass} && \multirow{2}{*}{$r\sqrt{\log n}$} & $(1/2, 3/4)$ 
    && \ding{55}           & \checkmark & \checkmark \\ \vspace{1.2em}
&&&$ (3/4, 1)$
    && \checkmark & \checkmark & \checkmark \\ \vspace{1.2em} 
Gaussian&&$r$&$ (1/2, 1)$
    && \checkmark & \checkmark & \checkmark \\ \vspace{.8em} 
Exponential&&$\displaystyle{\frac{r}{\sqrt{2\log n}}}$&$ (1/2, 1)$
    && \checkmark & \checkmark & \checkmark \\ \vspace{.7em} 
Student's $t_\nu$&&$\displaystyle{\frac{r\sqrt{2\log n}}{n^{(1-\beta)/\nu}}}$&$ (1/2, 1)$
    && \checkmark & \checkmark &    \ding{55}        \\
\bottomrule
\end{tabular}
\caption{Optimality of different tests for special cases of our asymptotic regime, where $\sigma_n$ is calibrated so that the problem is barely solvable. For the point mass, Gaussian and exponential distribution, a checkmark \checkmark indicates that the test 
achieves the optimal ``detection boundary" for the parameter $r$. For Student's $t_\nu$, there exists no sharp ``detection boundary" for $r$, and a checkmark \checkmark indicates that the test has full asymptotic power as $r\to\infty$. These results are proved in Theorems \ref{thm:main}--\ref{thm:polytail} and Corollary \ref{cor:density:pareto}.}
\label{table:main}
\end{table*}

We hope that our results will help to rehabilitate the max test, which enjoys many practical advantages over its special-purpose competitors in settings where asymptotic results are equivocal: First, its Type I error control is fairly robust to incorrect specification of the dependence between $p$-values; by contrast, the higher criticism test can be highly anticonservative even with very slight correlations between $p$-values. Second, when the max test rejects, the logical and mathematical basis for rejection is extremely simple and transparent: namely, that one $|X_i|$ value was too large. This simplicity confers a form of scientific robustness, allowing non-expert users to more easily interrogate how modeling assumptions contribute to the scientific conclusion. 
Third, beyond the multiple testing interpretation giving rise to the max test, we can also easily invert it to obtain a simple rectangular confidence region for $\mu\in \mathbb{R}^n$ giving simultaneous confidence intervals for every $\mu_i$; the totality of these inferences is much more informative than a binary accept/reject decision about the global null. By contrast, for the other tests, there is a more complex relationship between rejecting the global null and making inferences about individual $\mu_i$ values. Fourth, the modified higher criticism test cannot reject unless the fifth-largest $|X_i|$ is quite large; as a result, it is essentially powerless in the sparsest setting, where there are one or two extremely large signals. Finally, the max test is computationally cheap while the others require lengthy simulations.


\subsection{Related work}

Some recent theoretical work on global testing has relaxed the assumption of identical non-null means. \citet{tony2011optimal} considered the case where the non-null means are sampled from a Gaussian distribution $N(A_n, \sigma^2)$ where the variance $\sigma^2$ is fixed and $A_n = \sqrt{2r\log n}$ for some $r\in (0,1)$. Under this model, they showed that the higher criticism test achieves optimal asymptotic behaviour for $\beta\in (0, 1)$.
 Although different from a point mass, the model resembles the one in \citet{donoho2004higher}: since $\sigma^2$ is fixed as $n\to\infty$, the non-null means still concentrate around $\sqrt{2r \log n}$, leading to qualitatively similar limiting behavior as a point mass.
\citet{cai2014optimal} expanded this analysis to the more general model \eqref{eq:sigmaG}, proving optimality in certain conditions for the higher criticism test but not discussing the power of the commonly used max test.


The higher criticism's favorable theoretical performance has led to many efforts to generalize it beyond the model with independent errors studied here. One line of theoretical work has focused on studying the properties of higher criticism  type tests when observations are correlated.
\citet{hall2008properties} gave a detailed discussion of related issues. They showed that the null distribution of higher criticism  changes dramatically under weak dependence. In contrast, the max test is more robust to dependence, and the Type-I error can be controlled under arbitrary dependence.
 \citet{hall2010innovated} later proposed the innovated higher criticism to deal with the case of known covariance matrix with polynomially decaying off-diagonal elements. 
However, the innovated higher criticism can only be used if the covariance matrix of observations can be estimated reasonably well. Statisticians have also proposed various
extensions of higher criticism type tests for more general settings, such as ANOVA \citep{arias2011global}, time-frequency analysis \citep{cai2016global}, genetic association studies \citep{barnett2017generalized}, multi-sample analysis \citep{chan2015optimal}, and polynomial tailed noise distributions \citep{arias2019detection}, etc. It is an interesting question for future work whether the max test or generalizations thereof might perform equally well.

There has also been a lot of work that studies higher criticism type tests from a computational perspective.
In practice, the cutoff and $p$-values of higher criticism type statistics is often obtained by Monte Carlo simulation. An alternative approach for small sample size via numerical recursion was given by \citet{noe1972calculation}, \citet{owen1995nonparametric} and further developed by \citet{moscovich2013calibrated, moscovich2016exact} and \citet{li2015higher}. \citet{li2015higher} showed that their approximations for the $p$-value of higher criticism type statistics are reasonably accurate, even for small $p$-values and large samples.


Most papers on the global testing problem focus on the performance of the higher criticism or related statistics. Our contributions differ from these in that we show the max test enjoys many of the same theoretical advantages despite its simple form, and has similar finite sample power as the higher criticism test in a wide range of settings.


\section{Main results}

\subsection{The critical sparsity level}
We consider the following sequence of alternatives
\begin{equation}
\label{eq:althypo}
H_1^n: \mu_i\stackrel{i.i.d}{\sim} (1-\pi_n)\delta_0(\cdot)+\pi_n G_n({\cdot}),
\end{equation}
where the expected proportion of nonzero means is
$$\displaystyle \pi_n = n^{-\beta}, 0<\beta<1$$
and $G_n(\mu)$ is the distribution of the nonzero means. With slight abuse of notation, we will also use $G_n$ to denote the cumulative distribution function of the distribution. The alternative hypothesis in \citet{donoho2004higher} is a special case of this model taking $G_n$ as the point mass at $\sqrt{2r\log n},$ for $0<r \leq 1$. Following most previous literature on this topic, we restrict our attention to the sparse regime with $\beta < 1/2$; otherwise the $\chi^2$ test is potentially much more powerful than other tests. For simplicity, we drop the superscript on $H_1$ when the dimension $n$ is clear. 

The {\em total variation} (TV) distance between two probability measures $Q_1$ and $Q_2$ is defined as
$d_{TV}(Q_1, Q_2) = \sup_{A}|Q_1(A)-Q_2(A)|.$
For any test that tries to distinguish $H_1^n$ from $H_0^n$, the sum of its Type I and Type II error is lower bounded by
\[
1 - d_{\text{TV}}(H_0^n, H_1^n),
\]
where we write
$d_{\text{TV}}(H_0^n, H_1^n)$ as a shorthand for
\[
d_{\text{TV}}\left(H_0^n, H_1^n) = d_{\text{TV}}(\Phi^n, ((1-\pi_n)\Phi + \pi_n (G_n * \Phi))^n\right).
\]

By the Neyman-Pearson lemma \citep{neyman1933ix}, the likelihood ratio test is uniformly most powerful for testing $H_1^n$ against $H_0^n$. Indeed, the above lower bound is achieved if we reject $H_0^n$ when the likelihood ratio is greater than 1. Therefore, the TV distance $d_{\text{TV}}(H_0^n, H_1^n)$ tightly characterizes the hardness of the testing problem.

For any sequence $G_n$, the TV distance $d_{\text{TV}}(H_0^n, H_1^n)$ is non-increasing in $\beta$ for each $n$, with larger values of $\beta$ making the testing problem harder. Following \cite{cai2014optimal}, we introduce the concept of the {\em critical sparsity level}, which is a value $\beta^*$ that demarcates a sharp transition from asymptotic consistency to asymptotic powerlessness:

\begin{definition}\label{def:crit-sparse}
Fixing the sequence $\{G_n\}$, we define
\[
\underline{\beta}^* = \sup\left\{\beta \geq 0: \lim_n d_{\text{TV}}(H_0^n, H_1^n) = 1\right\}; 
\;\text{ and } \;\bar{\beta}^* = \inf\left\{\beta \leq 1: \lim_n d_{\text{TV}}(H_0^n, H_1^n) = 0\right\}.
\]
When $\underline{\beta}^* = \bar{\beta}^*$, we denote the common value as $\beta^*$, and call it the critical sparsity level corresponding to $\{G_n\}$.
\end{definition}

If a critical sparsity level $\beta^*$ exists for a sequence $\{G_n\}$ (i.e., if $\underline{\beta}^* = \bar{\beta}^*$), it follows from Definition~\ref{def:crit-sparse} that
\begin{itemize}
    \item If $\beta > \beta^*$, then $\lim_{n\to\infty} \EE_{H_1^n}[\phi_n(X)]\to \alpha$ for any sequence of level-$\alpha$ tests $\phi_n$, and
    \item If $\beta < \beta^*$, then $\lim_{n\to\infty} \EE_{H_1^n}[\phi_{LRT}(X)]\to 1$ for the level-$\alpha$ likelihood ratio test $\phi_{LRT}$.
\end{itemize}

We say that a sequence of level-$\alpha$ tests $\phi_n$ is {\em asymptotically consistent} on the sequence $\{H_1^n\}$ if  $\lim_{n\to\infty} \EE_{H_1^n}[\phi_n(X)]\to 1$ for any $\alpha$, and {\em asymptotically powerless} on the sequence $\{H_1^n\}$ if $\lim_{n\to\infty} \EE_{H_1^n}[\phi_n(X)]\to \alpha$ for any $\alpha$. \WFcomment{tweak this so it makes sense.}
We say that the sequence {\em achieves the optimal critical sparsity level} for the sequence $\{G_n\}$ if it has full asymptotic power whenever $\beta < \beta^*$. \WFcomment{Your plan makes sense to me, but "achieves the optimal critical sparsity level" sounds a bit clunky. Can we come up with a more elegant-sounding name?}


It will often be natural to parameterize the tail of Gaussian distribution as $\sqrt{2\delta \log n} \approx z_{n^{-\delta}}$, the upper $n^{-\delta}$ quantile of the standard normal distribution. If we define
\begin{equation}
\label{eq:taun}
\tau_n(\delta) = \log_n \mathbb{P}_{\mu \sim G_n}(X > \sqrt{2\delta\log n})
\end{equation}
as the tail probability of a single non-null observation, \cite{cai2014optimal} proved sufficient conditions for optimality of the higher criticism test in the sparse regime:

\begin{proposition}
\label{prop:HCoptimal}
 Suppose that $\{\tau_n(\delta)\}_{n=1}^{\infty}$ converges uniformly for all $\delta\in [0,1]$. Then the sequence of alternatives in \eqref{eq:althypo} has a critical sparsity level $\beta^*$, and if $\beta^* > 1/2$ then the level-$\alpha$ higher criticism test has full asymptotic power whenever $\beta < \beta^*$.
\end{proposition}

While \cite{cai2014optimal} only explicitly proved this for the higher criticism test, one can slightly modify their proof to show that this proposition holds for the modified higher criticism test and the Berk-Jones test as well (the proof is deferred to the Appendix). In this paper, we are interested in the following question: for which distributions $G_n$ does the max test achieve the same critical sparsity level $\beta^*$? \cite{donoho2004higher} showed that 
this is true when $G_n$ is a point mass and $\beta^*\geq 3/4$, which is by far the best-known result for this problem. We will show that under a mild regularity condition, when $\beta^*>1/2$, the max test also achieves the optimal critical sparsity level.

\subsection{Optimality of the max test}

To formally state our main result, we first need to introduce {\em regularly varying functions}. Following \cite{bingham1989regular}, we say that a function $Q: (0, \infty)\to (0, \infty)$ is a regularly varying function if
the limit 
$$g_Q(t) = \lim_{x\to\infty}\frac{Q(tx)}{Q(x)}$$
is finite and nonzero for all $t>0$. For any regularly varying function $Q$, it was shown in \cite{galambos1973regularly} that the limit $g_Q(t)$ has the form 
$$g_Q(t) = t^{\gamma}$$
for some value $\gamma\in(-\infty, \infty)$, which is called the {\em index of regular variation} of $Q$.

Among distributions with unbounded support, we consider those for which
$$-\max\{\log(1-G(\theta)), \log G(-\theta)\} \text{ is a regularly varying function}.$$
As noted by \cite{arias2019detection}, this class of distributions extended the definition of generalized Gaussian models, which are commonly used as benchmarks in this line of work. It covers the cases where $\log(1-G(\theta))=\log G(-\theta) \sim -\theta^{a}(\log \theta)^b, a>0, b\in\mathbb{R}$. The index $\gamma$ corresponds to the tail of the distribution $Q$, with smaller $\gamma$ indicating heavier tails. In particular, $\gamma = 2$ corresponds to a Gaussian tail, and $\gamma=1$ to an exponential tail.

Our main result shows essentially that the max test achieves the optimal detection boundary as long as $\beta \geq 3/4$ {\em or} the tail of $G$ is no lighter than Gaussian:

\begin{theorem}
\label{thm:main}
Under the assumptions of Proposition \ref{prop:HCoptimal}, suppose that either

(A1) $\beta ^* > 3/4$,  or

(A2) $G_n$ is a scale family with $G_n(\mu) = G(\mu/\sigma_n)$ for some sequence $\sigma_n$, where $-\max\{\log(1-G(\theta)), \log G(-\theta)\}$ is a regularly varying function with index of regular variation $\gamma\leq 2$.

Then if $\beta^* > 1/2$, the level-$\alpha$ max test $\phi_{\text{Max}}$ has full asymptotic power whenever $\beta < \beta^*$.
\end{theorem}

We will provide intuition for Theorem~\ref{thm:main} and a partial proof in Section~\ref{subsec:excesstail}, deferring a key technical lemma to the Appendix. The regularly varying assumption cannot be removed for $\beta^*<3/4$; see the Appendix for a counterexample where $G$ is stochastically larger than an exponential distribution but the max test is not optimal. Finally, note that neither Proposition~\ref{prop:HCoptimal} nor Theorem~\ref{thm:main} characterizes what occurs at the boundary where $\beta = \beta^*$; we discuss this boundary regime in the polynomial tail case in Section~\ref{subsec:polytail}.

Viewing the results of \citet{donoho2004higher} in light of Theorem~\ref{thm:main}, we see that the suboptimality of the max test in their asymptotic regime is a result of the assumption that all nonzero $\mu_i$ are identical. As a direct corollary of Theorem~\ref{thm:main}, we can derive explicit formulae for the critical sparsity levels of densities with polynomial tails, exponential tails, and Gaussian tails respectively: 




\begin{cor}
\label{cor:density:pareto}
Suppose that $G_n$ belong to a scale family with $G_n(\mu) = G(\mu/\sigma_n)$, for some distribution $G$ with density function $g(\theta)$.

(1) If $g(\theta) = \Theta(\theta^{-\nu-1})$ for some $\nu>0$, and $\sigma_n\sim n^\rho$ with $\rho > -(2\nu)^{-1}$, then the critical sparsity level is
$$\beta^*(\rho) = \nu \rho + 1.$$

(2) If $g(\theta) = \Theta(e^{-\theta})$ and $\sigma_n=r(2\log n)^{-1/2}$ with $r > \sqrt{2}/(\sqrt{2}-1)$, then 
$$\beta^*(r) = \left(1-\frac{1}{r}\right)^2.$$

(3) If $g(\theta) = \Theta\left(e^{-\frac{(\theta-\gamma)^2}{2}}\right)$ for some $\gamma\in\mathbb{R}$,  and $\sigma_n=r$ with $r > 1$, then
$$\beta^*(r) = \frac{r^2}{r^2 + 1}.$$

\end{cor}

\rev{\cite{cai2014optimal} derived the critical sparsity level when the alternative means follow the generalized Gaussian distribution, and Part (2) and (3) of this corollary are special cases of such distribution. In these two scenarios, Theorem \ref{thm:main} shows that the likelihood ratio test, the max test and the higher criticism test are asymptotically consistent when $\beta < \beta^*(r)$, and asymptotically powerless when  $\beta > \beta^*(r)$. As such, the critical sparsity level produces a sharp detection boundary for the scale parameter $r$. For example, if $g(\theta) = \Theta(e^{-\theta})$ and $\sigma_n=r(2\log n)^{-1/2}$, then all three tests are asymptotically consistent if $r> r^*(\beta):=(1-\sqrt{\beta})^{-1}$, and asymptotically powerless if $r< (1-\sqrt{\beta})^{-1}$. Thus, the desired sharp detection boundary is  $r^*(\beta) = (1-\sqrt{\beta})^{-1}$. Part (1) of this corollary exhibits a different regime: when the alternative means follow a $t$ distribution, there does not exist a sharp detection boundary for a scale parameter $r$. Instead, there is a sharp detection boundary in the growth rate $\rho$ if we set $\sigma_n = n^{\rho}$. We explore the boundary regime of the polynomial tail case further in Section~\ref{subsec:polytail}}.

\subsection{Proving Theorem \ref{thm:main} using excess tail values}
\label{subsec:excesstail}

In this section, we will explain the mathematical intuition behind Theorem \ref{thm:main}, and provide a sketch of its proof.
We begin by introducing a useful transformation of the empirical distribution of $X_i$ values, in terms of the tail parameter $\delta$. Defining \rev{$N(\delta) = \#\left\{i:\;\ |X_i| > \sqrt{2\delta\log n}\right\}$}, the higher criticism statistic may be rewritten as
\[
\text{HC}_n = \sup_{\delta \geq 0} \frac{N(\delta) - \mathbb{E}_0 N(\delta)}{\sqrt{\text{Var}_0 N(\delta)}}
    \approx \sup_{\delta > 0} \frac{N(\delta) - n^{1-\delta}}{n^{(1-\delta)/2}},
\]
where the approximation holds for large $n$, if the supremum is not achieved too close to $\delta=0$. Roughly speaking, then, the higher criticism test will have high power when the number of excess tail values is much larger than $n^{(1-\delta)/2}$, for some $\delta > 0$. By contrast, the max test rejects roughly when $N(1) > 0$.

Under the alternative, the most likely source of these excess tail values is the $n\pi_n$ non-null observations. We quantify their contribution as $N_1(\delta) = \#\left\{i:\; \mu_i \neq 0, |X_i| > \sqrt{2\delta \log n}\right\}$, and define
\begin{equation}
\label{eq:lambdan}
\lambda_n(\delta) = \log_n \mathbb{E}_1 N_1(\delta) = 1 - \beta + \tau_n(\delta),
\end{equation}
where $\tau_n$ is defined in Equation \ref{eq:taun}. Continuing our intuition from above, we expect that the higher criticism test will have high power when $\lambda_n(\delta) > \frac{1-\delta}{2}$ for any $\delta \in (0,1]$, while the max test will have high power when $\lambda_n(1) > 0$ in the limit.

\rev{Suppose that $\{\tau_n(\delta)\}_{n=1}^{\infty}$ converges uniformly for all $\delta\in [0,1]$. This is the same condition as Proposition~\ref{prop:HCoptimal} and Theorem~\ref{thm:main}. Under this condition, we denote 
\[
\tau^*(\delta) = \lim_{n\to\infty}\tau_n(\delta), \quad \text{and} \quad \lambda^*(\delta) = \lim_{n\to\infty}\lambda_n(\delta).
\]}
We can formalize the above heuristic characterization in Proposition \ref{prop:lambda-power}:
\begin{proposition}
\label{prop:lambda-power}
Suppose that $\beta>1/2$. Then
\begin{enumerate}[(a)]
\item If 
\rev{\[
\sup_{\delta \in (0,1]} \left[\lambda^*(\delta) - \frac{1-\delta}{2}\right] < 0,
\]}
then $d_{\text{TV}}(H_0^n, H_1^n) \to 0$. \WFcomment{Is this also true when $\beta^* < 1/2$?} 
\item If
\[
\sup_{\delta \in (0,1]} \left[\lambda^*(\delta) - \frac{1-\delta}{2}\right] > 0,
\]
then the likelihood ratio test, the higher criticism test, modified higher criticism test, and Berk-Jones tests all enjoy full asymptotic power.
\item The max test is asymptotically powerless if $\lambda^*(1) < 0$, and enjoys full asymptotic power if $\lambda^*(1) > 0$.
\end{enumerate}
\end{proposition}

The proof of part (a) of Proposition \ref{prop:lambda-power} is given in \cite{cai2014optimal}. \cite{cai2014optimal} also proved that the higher criticism test enjoys full asymptotic power under the condition of Part (b). For the modified higher criticism test and Berk-Jones tests, the proof is similar and is given in the Appendix for completeness.
Part (c) follows directly from the first and second Borel-Cantelli Lemma.

Proposition~\ref{prop:lambda-power} leaves open the question of what happens in the boundary regime where the supremum converges to 0. Indeed, Section~\ref{subsec:polytail} studies a natural regime with polynomial tails where $\lambda_n(1) \to 0$ and the modified higher criticism test is powerless in the limit even while the other tests enjoy full asymptotic power. 

Note further that the sufficient condition in Theorem~\ref{prop:lambda-power} for the max test to have full asymptotic power is more restrictive than the sufficient condition for the other three. This analysis suggests a disadvantage for the max test, which we illustrate in Figure~\ref{fig:lambda} showing four different $\lambda$ curves plotted against $\frac{1-\delta}{2}$. The black curve takes $G_n$ as a point mass, and shows a bad case for the max test: it rises above $\frac{1-\delta}{2}$ for a range of $\delta$ values that exclude 1. The other three curves, however (taking $G_n$ as Gaussian, exponential, and Cauchy), all show cases where the supremum is achieved at $\delta=1$, so that all of the tests enjoy high power.

Roughly speaking, max test achieves the optimal critical sparsity level if the supremum of $\displaystyle \lambda^*(\delta) - \frac{1-\delta}{2}$ is achieved at $\delta=1$. The following technical lemma connects this supremum with the tail property of $G_n$, and is essential in the proof of Theorem \ref{thm:main}.

\newtheorem{lemma}{Lemma}
\begin{restatable}{lemma}{lemtailbehavior}
\label{lemma:tailbehavior}
\begin{enumerate}[(a)]
    \item For any $\beta>1/2$ and sequence $\{G_n\}$,
    \[
    \sup_{\delta \in (0,1]} \left[\lambda^*(\delta) - \frac{1-\delta}{2}\right] \;\leq\; \max\left\{\lambda^*(1), \;\frac{3}{4}-\beta\right\}
    \]
    \item Under Assumption (A2) of Theorem \ref{thm:main}, 
    \[
    \sup_{\delta \in (0,1]} \left[\lambda^*(\delta) - \frac{1-\delta}{2}\right] \;\leq\; \max\left\{\lambda^*(1), \;\frac{1}{2}-\beta\right\}. 
    \]
\end{enumerate}
\end{restatable}

The proof of the lemma is given in Section \ref{sec:proof}. Theorem \ref{thm:main} is then a direct result of Lemma \ref{lemma:tailbehavior} and Proposition \ref{prop:lambda-power}.

\begin{proof}[Proof of Theorem \ref{thm:main}]

First, if $\beta < \beta^*$, then by definition of $\beta^*$,
\begin{equation}
\label{eq:betadiff}
\sup_{\delta \in (0,1]} \left[\lambda^*(\delta) - \frac{1-\delta}{2}\right] \geq \beta^* - \beta > 0.
\end{equation}

Suppose that $\beta^* > 3/4$ and $\beta < \beta^*$. Since the power of the max test is non-increasing in $\beta$, we can assume without loss of generality that $\beta > 3/4$. Since 
$3/4-\beta < 0$, we can combine \eqref{eq:betadiff} with part 1 of Lemma~\ref{lemma:tailbehavior} to conclude that $\lambda^*(1)>0$, implying that the max test has full asymptotic power. If Assumption (A2) of Theorem~\ref{thm:main} holds, then we can repeat the same argument replacing 3/4 with 1/2 and applying part 2 of Lemma~\ref{lemma:tailbehavior} instead of part 1.

\end{proof}

\begin{figure}
    \centering
    \includegraphics[width=.6\textwidth]{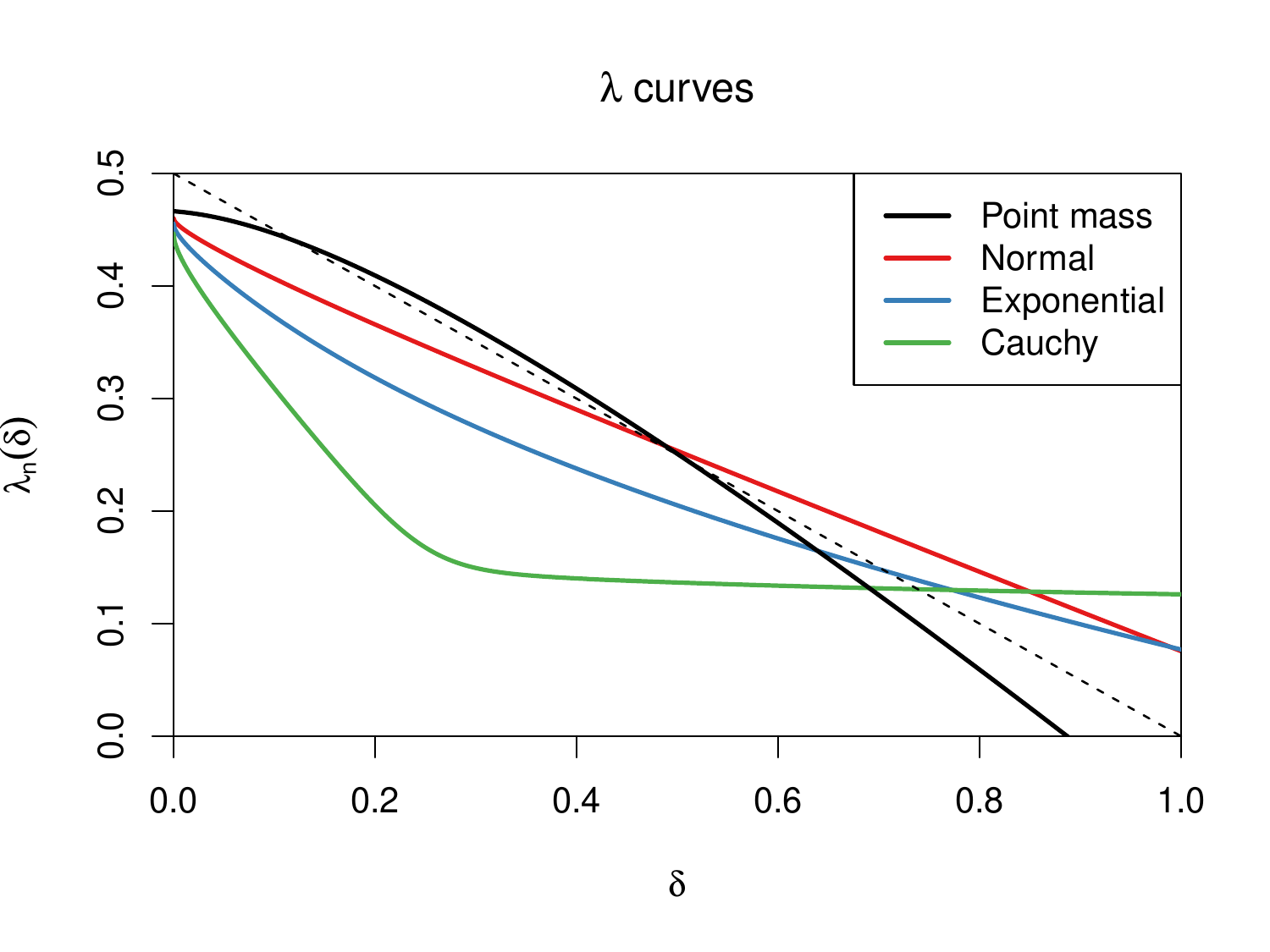}
    \caption{$\lambda_n(\delta)$ curves plotted against $\frac{1-\delta}{2}$ for four different tests, for $\lambda_n(\delta)$ as defined in \eqref{eq:lambdan}. If a curve rises above $\frac{1-\delta}{2}$ for some values of $\delta$ that exclude 1, then the likelihood ratio test has full asymptotic power, while the max test does not. The black curve, which takes $G_n$ as a point mass, shows this scenario. The other curves show case where the supremum $\lambda_n(\delta)-\frac{1-\delta}{2}$ of is achieved at $\delta=1$, so that the max test, the higher criticism test and the likelihood ratio test all enjoy full asymptotic power.}
    \label{fig:lambda}
\end{figure}


\subsection{Power analysis for polynomial tails}
\label{subsec:polytail}

Theorem \ref{thm:main} does not characterize the power of different tests in the boundary regime. We now study a natural regime with polynomial tails, with $\beta = \beta^*$. 
The boundary regime with polynomial tails is more interesting because we have shown in Corollary \ref{cor:density:pareto} that there is not a sharp detection boundary for a scale parameter, but rather in the growth rate $\rho$ where $\sigma_n = n^{\rho}$.

In this section we explore a sequence of alternative distributions growing at the critical rate $\rho$, and parametrized by a scale parameter $r$. Under this sequence of alternatives, we will show that the asymptotic power of level-$\alpha$ max test is a smooth function of $r\in(0, \infty)$, and converges to 1 as $r\to\infty$.
In addition, we will show that the modified higher criticism test is asymptotically powerless no matter what $r$ is.

Suppose that $G_n(\mu) = G(\mu/\sigma_n$), where $G$ is the $t$ distribution with $\nu$ degrees of freedom. Then the density function $g(\theta)=\Theta(\theta^{-\nu-1}).$ Recall from Corollary \ref{cor:density:pareto} that if $\liminf_n\log_n\sigma_n > (\beta-1)/\nu$, then the max test and higher criticism test both have full asymptotic power. If $\limsup_n\log_n\sigma_n < (\beta-1)/\nu$, then both tests are powerless. Therefore, to study the boundary regime, we are interested in the case where $\lim_n\log_n\sigma_n = (\beta-1)/\nu$.
 Fix $\beta\in (1/2, 1)$, and let
$$\sigma_n=\frac{r\sqrt{2\log n}}{n^{(1-\beta)/\nu}},r\in(0, \infty).$$ 
Then it can be verified that the power of the max test has smooth transition from $\alpha$ to 1 as $r$ goes from 0 to $\infty$. The higher criticism test also shares this smooth transition behavior, as the rejection threshold for $p_{(1)}$ in the higher criticism statistic
is very close to $\alpha/n$.
Perhaps surprisingly, the modified higher criticism test is asymptotically powerless in this case, as detailed by the following theorem.


\begin{restatable}{theorem}{polytail}

\label{thm:polytail}
Suppose that $G$ satisfies $\lim_{\mu\to\infty}(1-G(\mu))\mu^{\nu}= \lim_{\mu\to\infty}G(-\mu)\mu^{\nu}= C$ with tail index $\nu>0$, and $\displaystyle \sigma_n=\frac{r\sqrt{2\log n}}{n^{(1-\beta)/\nu}}$ for some $\beta\in (1/2, 1)$. Then $\beta^* = \beta$, and
\begin{enumerate}
    \item the asymptotic power of the level-$\alpha$ max test, is
    \[
    \lim_{n\to\infty}\mathbb{P}_{H_1}(\text{reject }H_0) = 1-e^{-2C r^{\nu} + \log(1-\alpha)}.
    \] 
     In particular, the power tends to 1 as $r\to\infty$.
    \item for any $r\in (0, \infty)$, the modified higher criticism is asymptotically powerless.
\end{enumerate}
\end{restatable}

We note that for fixed $r$, the power of the max test as $n$ goes to infinity does not depend on the sparsity parameter $\beta$. This is because $\sigma_n$ is a decreasing function of the sparsity level $\beta$, thereby implicitly adjusting for the sparsity level. 

Compared to the original higher criticism test, the modified higher criticism test was designed to ignore $p-$values smaller than $1/n$. These small $p$-values cause the original higher criticism statistics to have a heavy right tail under the null distribution, and the modified higher criticism test is considered in \cite{donoho2004higher} as a refined test with potentially better finite sample performance. However, this modification also makes the modified higher criticism test powerless in situations where the smallest $p$-values provide the best evidence against the null. Recall that $\lambda_n(\delta)$ is defined as the log of the expected number of non-null observations that are greater than $\sqrt{2\delta\log n}$.
In Theorem \ref{thm:polytail}'s setting, the proof of Corollary~\ref{cor:density:pareto} shows that
$\lambda^*(\delta) = 0$ for all $\delta\in (0, 1]$; as a result $\lambda^*(\delta) < (1-\delta)/2$ for all $\delta < 1$.
In other words, evidence against the null is only present in the number of tail values exceeding $\sqrt{2\log n}$, which is roughly the Bonferroni threshold. Because the $p$ values of these observations are smaller than $1/n$, they are effectively truncated by the modified higher criticism test, making it asymptotically powerless. The original higher criticism test, however, is still powerful because, like the max test, it can reject on the strength of the largest $p$-value alone. A full proof of Theorem~\ref{thm:polytail} is given in the Appendix.


\section{Numerical results}
\label{sec:simulation}

We now provide simulation results showing that the max test has similar power as the higher criticism test when the distribution of non-null signals has Gaussian or heavier tails. We generate data under the following alternative:
\begin{align*}
X_i\simind N(\mu_i, 1), \mu_i\simiid G_n, \quad  &\text{for } i=1, \dots, n_1\\
X_i\simind N(0, 1),\quad &\text{for } i=n_1+1, \dots, n.
\end{align*}
In this section, we consider the case where $G_n$ has either exponential or Cauchy tail. In the appendix, we provide additional simulation results for other distributions $G_n$, including Gaussian distribution. We take $n=50,000$ and $n_1=\lfloor n^{1-\beta}\rfloor$, where the sparsity parameter $\beta$ ranges from 0.1 to 0.9. We compare the power of the following 6 tests: the max test, the higher criticism test, the modified higher criticism test, the Berk-Jones test, the $\chi^2$ test and a hybrid test which combines the max test and the $\chi^2$ test. The rejection region of the level $\alpha$ hybrid test has the form
\[
\left\{\max_{i}|X_i| > m(n, \alpha/2)\right\}\cup\left\{\sum_{i}X_i^2 > c(n, \alpha/2)\right\},
\]
where $m(n, \alpha/2)$ and $c(n, \alpha/2)$ are the $1-\alpha/2$ quantiles of $\max_{i}|X_i|$ and $\sum_{i}X_i^2$ under the null. For all 6 tests, we control Type-I error at $\alpha = .05$. For the first five tests, we use the empirical 95\% percentile of the test statistics under the null distribution as the cutoff value; for the hybrid test, we use the empirical 97.5\% percentile of $\max_{i}|X_i|$ and $\sum_{i}X_i^2$ to estimate the threshold $m(n, \alpha/2)$ and $c(n, \alpha/2)$. Our results are summarized below.


\vspace{2mm}
\noindent{\textbf{When $\bm{G_n}$ has exponential tail}} In particular, we choose $G_n=\text{Laplace}(0, r)$. The power of all six tests are shown in Figure \ref{fig:comp}. First, we found that when $\beta\leq 0.3$, the $\chi^2$ test (yellow curve) outperforms all five others, and the max test is least powerful due to relatively dense signals.
Second, the modified higher criticism test has very low power when $\beta>0.5$.
Since the modified higher criticism test does not use the $p$-values smaller than $1/n$, it performs subpar in the sparse regime where the max test and the higher criticism test reject the null based on those $p$-values. 
Third, when $\beta>0.5$ the power of the max test, the higher criticism test and the Berk-Jones test are very similar. This finding agrees with our Theorem 1, which states that the max test achieves the optimal critical sparsity level for exponentially distributed alternatives when $\beta>0.5$.
Finally, the hybrid test, which combines the max test and the $\chi^2$ test, performs on par if not better than the higher criticism under all sparsity regimes. 



\begin{figure}
\centering
\includegraphics[width = .8\textwidth]{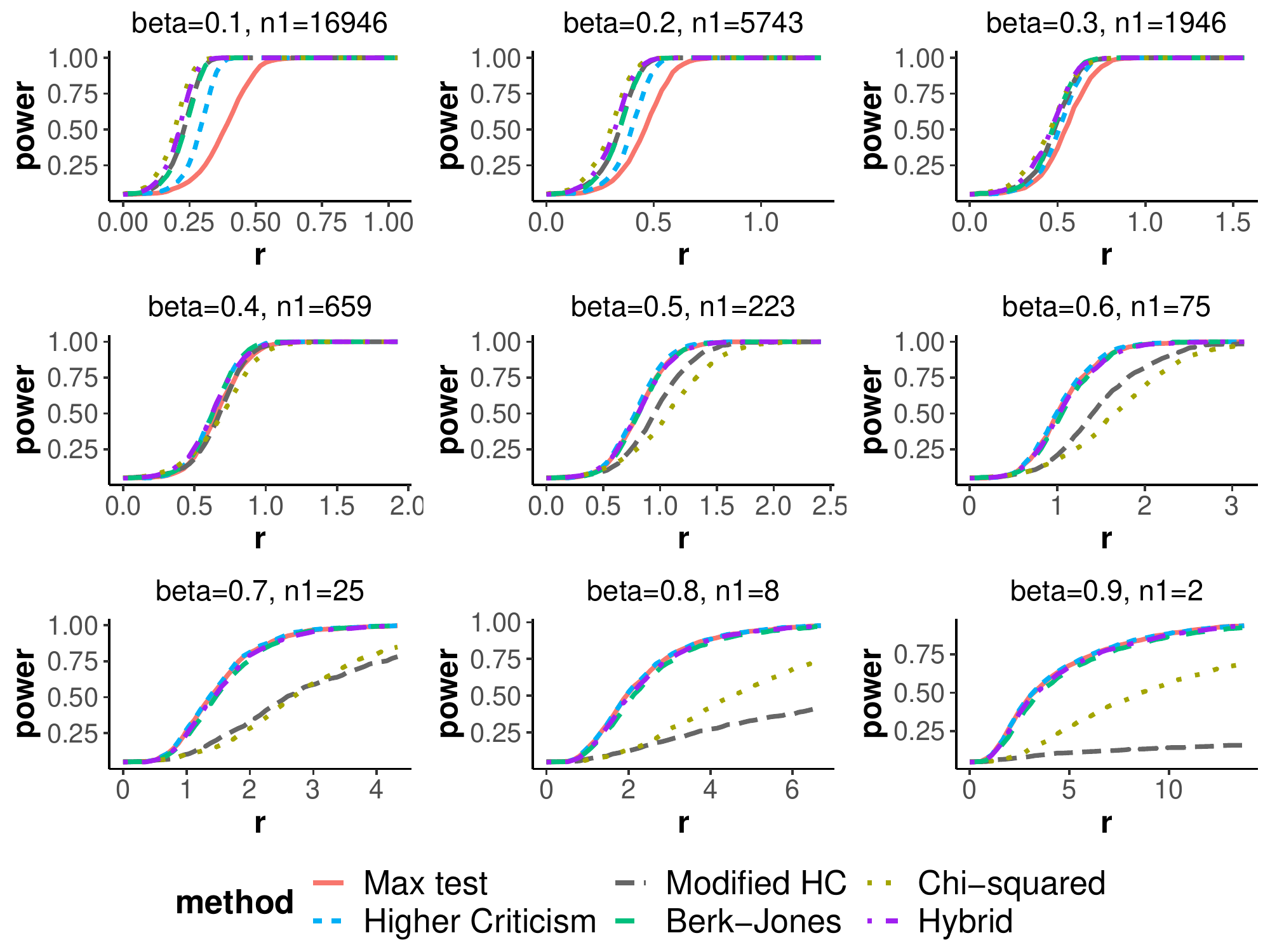}
\caption{Comparison of power for different methods: max test (red curve), higher criticism (light blue curve), modified higher criticism (grey curve), Berk-Jones (green curve) ,$\chi^2$ test (yellow curve) and the hybrid test (purple curve). Here $n=50,000$ with $n_1 = \lfloor n^{1-\beta}\rfloor$ non-null means drawn from $\text{Laplace}(0, r)$. The horizontal axis shows the value of $r$ while the vertical axis shows power. 
}
\label{fig:comp}
\end{figure}

\vspace{3mm}
\noindent{\textbf{When $\bm{G_n}$ has polynomial tail}} In particular, we choose $\displaystyle {G_n=}\text{Cauchy}{\left(0, \frac{r\sqrt{2\log n}}{n^{(1-\beta)}}\right)}$.  Recall that according to Theorem \ref{thm:polytail}, under this setting the max test and the higher criticism should have very high power when $r$ is big, while modified higher criticism should have little power. Indeed, the max test, the higher criticism test, the Berk-Jones test and the hybrid test have almost identical power for all combinations of $(\beta, r)$, and the modified higher criticism performs worst among all tests. All of these findings are consistent with our Theorem \ref{thm:main}.
We also notice that for fixed $r$ value, the power of max test, higher criticism and Berk-Jones are almost constant for different parameter $\beta$.  This finding also agrees with the asymptotic power of max test in Theorem \ref{thm:polytail}.

Appendix~\ref{app:additional-simulations} gives analogous simulations for Gaussian, logistic, $\chi^2(1)$, $t_3$, and $t_5$ distributions, with qualitatively similar results. Overall, our simulation confirms that the higher criticism test does not have better finite sample power than the max test when the max test achieves the optimal critical sparsity level. On the other hand, when the higher criticism does have better power over the max test, the non-null signals are likely dense enough such that the $\chi^2$ test is even more powerful.

\begin{figure}
\centering
\includegraphics[width = .8\textwidth]{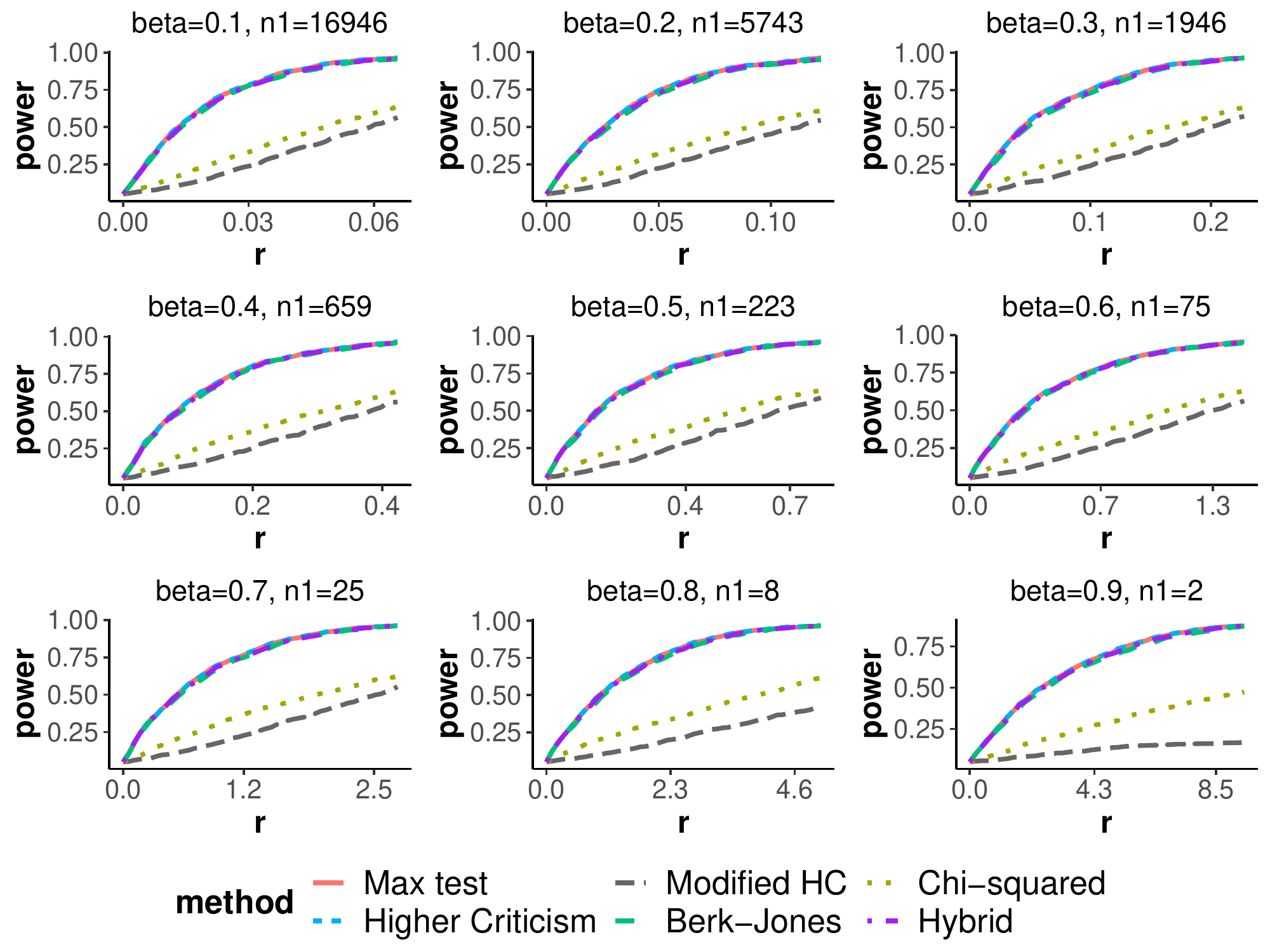}
\caption{Comparison of power for different methods, where  $n=50,000$ with $n_1 = \lfloor n^{1-\beta}\rfloor$ non-null means drawn from $\text{Cauchy}(0, r\sqrt{2\log n}n^{-(1-\beta)})$. The horizontal axis shows the value of $r$ while the vertical axis shows power. }
\label{fig:cauchy}
\end{figure}

\section{Discussion}

We have shown, theoretically and numerically, that the max test has optimal asymptotic behavior in the sparse regime, provided that the distribution of non-null signals has a tail no lighter than Gaussian. In addition, the max test dominates the modified higher criticism test when the distribution of nonzero signals has polynomial tails. We believe our results complicate the conventional wisdom that the max test is a substandard test for the purpose of signal detection and suggest that in many applied settings practitioners will not suffer low performance by using the max test. In these settings, the max test can be derived as a ``free" (incurring no additional FWER) deduction from simultaneous confidence intervals for the coordinates of $\mu_i$.

The higher criticism has been generalized to many interesting statistics problems beyond the signal detection problem studied here. It is an interesting question for future work whether in many of these cases it may be possible to find an analogous generalization of the max test whose performance matches the higher-criticism-type test.

Like other papers in this line of research, our paper did not address the ``weak, dense" regime, where the sparsity parameter $\beta$ is smaller than $1/2$. It is well known that in the dense regime, the $\chi^2$ test has higher power than the higher criticism and max test when the distribution of non-null means is a point mass. We have suggested a hybrid test based on combining the p-values of the $\chi^2$ and the max test, and shown numerical evidence that it performs well throughout the sparse and dense regimes. By inverting this hybrid test, we can obtain a joint confidence region for $\mu \in \mathbb{R}^n$ that is the union of an $\ell_2$ ball and an $\ell_\infty$ ball around the observed $X$, simultaneously giving short intervals for coordinates of $\mu_i$ and reasonable intervals for all linear combinations of $\mu$. Finding a test that achieves the optimal critical sparsity level under the general model in this regime is a interesting direction for future research.



\section{Proofs of main results}
\label{sec:proof}

We begin by proving the following result on the tail probability of $X\sim N(\mu, 1), $ where $\mu$ is generated from some distribution $G_n$. This is a standard result and is repeated here for completeness of the proof.
\begin{lemma}
\label{lemma:tail}
Let $\bar{G}_n(\theta) = 1-G_n(\theta)$. Under the alternative hypothesis \eqref{eq:althypo},  
we have
\[
\tau_n(\delta)=\sup_{0\leq t\leq 1}-\frac{Q_n(t\sqrt{2\delta\log n})}{\log n}-\delta(1-t)^2 + O\left(\frac{\log \log n}{\log n}\right),
\]
where $Q_n(\theta) = -\max\{\log \bar{G}_n(\theta), \log G_n(-\theta)\}$ and the $O\left(\frac{\log \log n}{\log n}\right)$ term is uniform over all $\delta \in [0,1]$.
\end{lemma}

\begin{proof}
For any $0\leq t, \delta \leq 1$, we have
\begin{align*}
\mathbb{P}_{\mu\sim G_n}\left(X\geq \sqrt{2\delta\log n}\right)
    &\;\geq\; \mathbb{P}_{\mu\sim G_n}\left(X\geq \sqrt{2\delta\log n}, \mu\geq t\sqrt {2\delta \log n}\right)\\
&\;=\; \mathbb{P}_{\mu\sim G_n}\left(\mu\geq t\sqrt{2\delta \log n}\right) \mathbb{P}_{\mu\sim G_n}\left(X\geq \sqrt{2\delta\log n}\;\mid\;\mu \geq t\sqrt{2\delta \log n}\right)\\
&\;\geq\; \left(1-G_n({t\sqrt{2\delta \log n}})\right) \left(1-\Phi((1-t)\sqrt{2\delta\log n})\right)\\
&\;\geq\; \frac{1}{6\sqrt{2\log n}}\,\exp\left\{\log\bar{G}_n\left(t\sqrt{2\delta\log n}\right)-(1-t)^2\delta\log n\right\}, 
\end{align*}
where the last inequality follows from the fact that
$1-\Phi(x)\geq \frac{1}{3(x+1)}e^{-x^2/2}$
for any $x>0$. Taking the supremum over $t\in [0, 1],$ we have
\[
\mathbb{P}_{\mu\sim G_n}\left(X\geq \sqrt{2\delta\log n}\right)\;\geq\; 
\frac{1}{6\sqrt{2\log n}}\,\exp\left\{\sup_{0\leq t \leq 1} \log\bar{G}_n\left(t\sqrt{2\delta\log n}\right)-(1-t)^2\delta\log n\right\}.
\]
On the other hand, Fubini's theorem yields
\begin{align}
&\mathbb{P}_{\mu\sim G_n}\left(X\geq \sqrt{2\delta\log n}\right) \nonumber\\
= \;\ & -\int_{-\infty}^{\infty}\bar\Phi(\sqrt{2\delta\log n}-\mu)d\bar{G}_n(\mu) \nonumber\\
= \;\ & -\int_{-\infty}^{0}\bar\Phi(\sqrt{2\delta\log n}-\mu)d\bar{G}_n(\mu) -\int_{0}^{\sqrt{2\delta \log n}}\bar\Phi(\sqrt{2\delta\log n}-\mu)d\bar{G}_n(\mu) \nonumber\\
&-\int_{\sqrt{2\delta \log n}}^{\infty}\bar\Phi(\sqrt{2\delta\log n}-\mu)d\bar{G}_n(\mu) \label{eq:fubinistep1}
\end{align}
For the first and third terms of Equation \ref{eq:fubinistep1}, we have
\begin{equation}
\label{eq:fubiniterm1}
    -\int_{-\infty}^{0}\bar\Phi(\sqrt{2\delta\log n}-\mu)d\bar{G}_n(\mu) \leq n^{-\delta}\bar{G}_n(0) 
\end{equation}
and
\begin{equation}
\label{eq:fubiniterm3}
    -\int_{\sqrt{2\delta \log n}}^{\infty}\bar\Phi(\sqrt{2\delta\log n}-\mu)d\bar{G}_n(\mu) \leq \bar G_n(\sqrt{2\delta\log n}).
\end{equation}
For the second term, we have
\begin{align}
    &-\int_{0}^{\sqrt{2\delta \log n}}\bar\Phi(\sqrt{2\delta\log n}-\mu)d\bar{G}_n(\mu) \nonumber \\
\leq\;\ &-\int_{0}^{\sqrt{2\delta\log n}}e^{-\frac{1}{2}(\sqrt{2\delta\log n}-\mu)^2}d\bar{G}_n \label{eq:fubinitail}\\
=\;\ &-\bar{G}_n(\mu)e^{-\frac{1}{2}(\sqrt{2\delta\log n}-\mu)^2}\bigg\rvert^{\mu=\sqrt{2\delta\log n}}_{\mu=0}+\int_{0}^{\sqrt{2\delta\log n}}(\sqrt{2\delta\log n}-y)\bar G_n({y})e^{-\frac{1}{2}(\sqrt{2\delta\log n}-y)^2}dy \label{eq:fubiniintegralbyparts}\\
=\;\ &n^{-\delta}\bar{G}_n(0) - \bar G_n(\sqrt{2\delta\log n}) + \int_{0}^{1}(2\delta\log n) (1-t)\bar G_n(t\sqrt{2\delta\log n})e^{-\frac{1}{2}(\sqrt{2\delta\log n}(1-t))^2}dt \label{eq:fubinichangeofvar}\\
\leq\;\ & n^{-\delta}\bar{G}_n(0) - \bar G_n(\sqrt{2\delta\log n}) + (2\log n)\exp\left\{\sup_{0\leq t \leq 1}\log\bar{G}_n\left(t\sqrt{2\delta\log n}\right)-(1-t)^2\delta\log n\right\} \label{eq:fubinisupintegrand}.
\end{align}
where Equation \ref{eq:fubinitail} is obtained by Gaussian tail bounds, Equation \ref{eq:fubiniintegralbyparts} by integration by parts, Equation \ref{eq:fubinichangeofvar} by changing of variables, and Equation \ref{eq:fubinisupintegrand} by taking the supremum of the integrand over $t\in[0, 1].$ Combining Equations \ref{eq:fubinistep1}, \ref{eq:fubiniterm1}, \ref{eq:fubiniterm3}, and \ref{eq:fubinisupintegrand}, we have
\[
\mathbb{P}_{\mu\sim G_n}\left(X\geq \sqrt{2\delta\log n}\right) \leq (2\log n+2)\exp\left\{\sup_{0\leq t \leq 1}\log\bar{G}_n\left(t\sqrt{2\delta\log n}\right)-(1-t)^2\delta\log n\right\}.
\]
Therefore,
\begin{equation}
    \begin{aligned}
    \label{eq:uppertailbound}
    & \frac{1}{6\sqrt{2\log n}}\,\exp\left\{\sup_{0\leq t \leq 1}\log\bar{G}_n\left(t\sqrt{2\delta\log n}\right)-(1-t)^2\delta\log n\right\} \\
   &\;\leq\; \mathbb{P}_{\mu\sim G_n}\left(X\geq \sqrt{2\delta\log n}\right) \leq (2\log n+2)\exp\left\{\sup_{0\leq t \leq 1}\log\bar{G}_n\left(t\sqrt{2\delta\log n}\right)-(1-t)^2\delta\log n\right\}.
    \end{aligned}
\end{equation}
Similarly, we have
\begin{equation}
    \begin{aligned}
    & \frac{1}{6\sqrt{2\log n}}\,\exp\left\{\sup_{0\leq t \leq 1}\log {G}_n\left(-t\sqrt{2\delta\log n}\right)-(1-t)^2\delta\log n\right\} \\
   &\;\leq\; \mathbb{P}_{\mu\sim G_n}\left(X\leq -\sqrt{2\delta\log n}\right) \leq (2\log n+2)\exp\left\{\sup_{0\leq t \leq 1}\log{G}_n\left(-t\sqrt{2\delta\log n}\right)-(1-t)^2\delta\log n\right\}.
    \end{aligned}
\end{equation}
Combining the two equations above, we have \WFcomment{Changed 6 to 3 below}
\begin{equation}
    \begin{aligned}
    & \frac{1}{3\sqrt{2\log n}}\,\exp\left\{\sup_{0\leq t \leq 1}- Q_n\left(t\sqrt{2\delta\log n}\right)-(1-t)^2\delta\log n\right\} \\
   &\;\leq\; \mathbb{P}_{\mu\sim G_n}\left(|X|\geq \sqrt{2\delta\log n}\right) \leq (4\log n+4)\exp\left\{\sup_{0\leq t \leq 1}-{Q}_n\left(t\sqrt{2\delta\log n}\right)-(1-t)^2\delta\log n\right\}.
    \end{aligned}
\end{equation}
Taking $\log_n$ on both sides, we have \WFcomment{Fixed the equation below too, I think}
\[
-\frac{\log(3\sqrt{2\log n})}{\log n} \leq \tau_n(\delta) - \left[\sup_{0\leq t\leq 1}-\frac{Q_n(t\sqrt{2\delta\log n})}{\log n}-\delta(1-t)^2\right] \leq \frac{\log(4\log n + 4)}{\log n}.
\]
We conclude that
\[
\tau_n(\delta)=-\sup_{0\leq t\leq 1}\frac{Q_n(t\sqrt{2\delta\log n})}{\log n}-\delta(1-t)^2 + O\left(\frac{\log \log n}{\log n}\right),
\]
where the $O\left(\frac{\log \log n}{\log n}\right)$ term is uniform over all $\delta\in [0, 1]$
\end{proof}

We are now ready to restate and prove Lemma \ref{lemma:tailbehavior}:

\lemtailbehavior*
\begin{proof}
Define 
$$g_n(\delta, t)=-\frac{Q_n(t\sqrt{2\delta\log n})}{\log n}+\delta\left[\frac{1}{2}-(1-t)^2\right], \quad \text{and } \quad h_n(\delta)=\sup_{0\leq t\leq 1} g_n(\delta, t).$$
Applying Lemma \ref{lemma:tail}, we have
\begin{align*}
    h_n(\delta)
    &=\tau_n(\delta) + \frac{\delta}{2} + O\left(\frac{\log \log n}{\log n}\right)\\
    &=\lambda_n(\delta) - \frac{1-\delta}{2} + \beta - \frac{1}{2} + O\left(\frac{\log \log n}{\log n}\right).
\end{align*}

To prove part (a), it suffices to show that
\[
h_n(\delta)\leq\max\{h_n(1), 1/4\}, \quad \text{ for all } \delta\in(0,1).
\]




We prove this claim by supposing that $h_n(\delta) > \max\{h_n(1), 1/4\}$ for some $\delta < 1$, and deriving a contradiction.

Let $\delta^*_n$ and $t^*_n$ be values that jointly maximize $g_n(\delta, t)$ over $0\;\leq\; \delta,t\;\leq\; 1$. By assumption,  
\begin{equation}
\label{eq:partabound1}
\frac{1}{4} < g_n(\delta^*_n,t^*_n) \leq \delta_n^*\left[\frac{1}{2} - (1-t_n^*)^2\right],
\end{equation}
so we must have $\delta^*_n>1/2$ and $t^*_n>1/2$, and also
\[
\frac{1}{4\delta_n^*} < \frac{1}{2} - (1-t_n^*)^2.
\]
Further, because $g_n(\delta^*_n,t^*_n) > h_n(1)$, we also have
\begin{align*}
0 &\;<\; g_n(\delta_n^*, t_n^*) - g_n(1, t_n^*\sqrt{\delta_n^*})\\
  &\;<\; \delta_n^*\left(\frac{1}{2} - (1-t_n^*)^2\right) - \left(\frac{1}{2} - \left(1-t_n^*\sqrt{\delta_n^*}\right)^2\right),
\end{align*}
which leads to 
\begin{align*}
&\delta_n^*\left(\frac{1}{2} - (1-t_n^*)^2\right) - \left(\frac{1}{2} - \left(1-t_n^*\sqrt{\delta_n^*}\right)^2\right) > 0\\
\iff &\frac{1}{2}\delta_n^*-\delta_n^*+2t_n^*\delta_n^*-\delta_n^*(t_n^*)^2-\frac{1}{2}+1+\delta_n^*(t_n^*)^2-2t_n^*\sqrt{\delta_n^*} > 0\\
\iff &\frac{1}{2} - \frac{1}{2}\delta_n^* + 2t_n^*\delta_n^* - 2t_n^*\sqrt{\delta_n^*} > 0\\
\iff & 2t_n^*\sqrt{\delta_n^*}(\sqrt{\delta_n^*}-1) > \frac{1}{2}(\sqrt{\delta_n^*}-1)(\sqrt{\delta_n^*}+1) \\
\iff &2t_n^*\sqrt{\delta_n^*} < \frac{1}{2}(\sqrt{\delta_n^*}+1) \\
\iff &\frac{1}{4\delta_n^*} > \left(2t_n^* - \frac{1}{2}\right)^2.
\end{align*}
Combining the two equations above, we have
$$\frac{1}{2}-(1-t^*_n)^2>\left(2t^*_n-\frac{1}{2}\right)^2,$$
a contradiction for $t^*_n>1/2$.

Turning to part (b), suppose that $Q_n(\theta) = Q(\theta/\sigma_n)$ for some sequence $\sigma_n$, where $Q(\theta)$ is a regularly varying function with $g_Q(a)\leq a^2$. 
We consider the following two scenarios.
\begin{enumerate}[(i)]
\item $\limsup \sqrt{2\log n}\sigma_n^{-1} < \infty$.

Since the distribution $G$ has unbounded support, and $\limsup \sqrt{2\log n}\sigma_n^{-1} < \infty$, $Q(\sqrt{2\log n}\sigma_n^{-1})$ is bounded. Therefore
\[
\lim_n \frac{Q(\sqrt{2\log n}\sigma_n^{-1})}{\log n} = 0,
\]
and
\[
\lambda^*(\delta)-\frac{1-\delta}{2} = \lim_n h_n(\delta)+\frac{1}{2}-\beta = \sup_{0\leq t\leq 1} \delta\left[\frac{1}{2}-(1-t)^2\right] +\frac{1}{2}-\beta = \frac{\delta}{2} +\frac{1}{2}-\beta.
\]
Hence the supremum of $\lambda^*(\delta)-\frac{1-\delta}{2}$ on $\delta\in[0, 1]$ is attained at $\delta=1$.
\item $\limsup \sqrt{2\log n}\sigma_n^{-1} = \infty$. 

Note that the limit $\lambda^*(\delta)$ exists for any $\delta$. Therefore, by considering the sub-sequence of $\sigma_n$ with $\sqrt{2\log n}\sigma_n^{-1} \to \infty$, we can assume without loss of generality that $\sqrt{2\log n}\sigma_n^{-1} \to \infty$.
To prove the desired inequality, it suffices to show that, for any $\epsilon > 0$, there exists $\bar{n}(\epsilon) \in \mathbb{N}$ such that
\[
h_n(\delta)\leq\max\{h_n(1), 0\} + \epsilon, \quad \text{ for all } \delta \in (0,1), \; n > \bar{n}(\epsilon).
\]
Fix $\epsilon > 0$. Like part (a), we will prove this by supposing that $h_n(\delta) > \max\{h_n(1), 0\} + \epsilon$ for all $n$ and some $\delta$, and deriving a contradiction. Suppose that for any $N>0$, there exists $n> N$ and $(\delta^*_n, t^*_n)\in [0,1)\times [0,1]$ such that
$$g_n(\delta^*_n, t^*_n) > \max\{\epsilon, h_n(1)+\epsilon\}.$$

To make use of the regularly varying property, we need to first obtain upper and lower bound for $t^*_n\sqrt{\delta^*_n}$. Since 
\begin{equation}
\label{eq:regbound1}
g_n(\delta^*_n, t^*_n) = \delta^*_n\left[\frac{1}{2}-(1-t^*_n)^2\right] - \frac{Q(t^*_n\sqrt{2\delta^*_n\log n}\sigma_n^{-1})}{\log n} >\epsilon
\end{equation}
and $Q$ is non-negative, we have $\delta^*_n>2\epsilon$ and $t^*_n>1-\sqrt{2}/{2}>1/4$. Therefore
$$\frac{1}{2t^*_n\sqrt{\delta^*_n}}<\sqrt{\frac{2}{\epsilon}}.$$
On the other hand, we have $g_n(\delta^*_n, t^*_n) > h_n(1) + \epsilon \geq g_n(1, 1/2) + \epsilon$, that is, 
\begin{equation}
\label{eq:regbound2}
\delta^*_n\left[\frac{1}{2}-(1-t^*_n)^2\right] - \frac{Q(t^*_n\sqrt{2\delta^*_n\log n}\sigma_n^{-1})}{\log n} > \frac{1}{4}+\epsilon - \frac{Q(\frac{1}{2}\sqrt{2\log n}\sigma_n^{-1})}{\log n}.
\end{equation}
Following the first claim (see Equation \ref{eq:partabound1}), we have $$\delta^*_n\left[\frac{1}{2}-(1-t^*_n)^2\right] \leq \frac{1}{4} < \frac{1}{4}+\epsilon.$$
Comparing the two equations above and noting that $Q$ is non-decreasing, we have
\[
\frac{Q(t^*_n\sqrt{2\delta^*_n\log n}\sigma_n^{-1})}{\log n} \leq \frac{Q(\frac{1}{2}\sqrt{2\log n}\sigma_n^{-1})}{\log n}.
\]
Therefore
$$t^*_n\sqrt{\delta^*_n} \leq \frac{1}{2}.$$

Using properties of regularly varying functions \citep{bingham1989regular}, we know that 
\[
\lim_{s\to\infty}\sup_{a\in\Gamma}\left|\frac{Q(as)}{Q(s)} -a^2\right|\to 0
\]
for any compact set $\Gamma$. Therefore, for any $c_0>0$ there exists $S>0$ such that 
$$\frac{Q(as)}{Q(s)} \leq a^2 + c_0 $$
for any $s>S$ and $a\in \left[1, \sqrt{\frac{2}{\epsilon}}\right].$
Take $s = t^*_n\sqrt{2\delta^*_n\log n}\sigma_n^{-1}$. Since $\sqrt{2\log n}\sigma_n^{-1} \to \infty$, we know that for large enough $n$,

\begin{equation}
\label{eq:regbound3}
    \frac{Q(\frac{1}{2}\sqrt{2\log n}\sigma_n^{-1})}{Q(t^*_n\sqrt{2\delta^*_n\log n}\sigma_n^{-1})} \leq \frac{1}{4t^{*2}_n\delta^*_n}+c_0.
\end{equation}
Combining Equations \ref{eq:regbound1}, \ref{eq:regbound2} and \ref{eq:regbound3}, we have
$$\left(\frac{1}{4t^{*2}_n\delta^*_n}+c_0\right)\delta^*_n\left[\frac{1}{2}-(1-t^*_n)^2\right] > \frac{1}{4}+\epsilon.$$
Since $c_0$ is arbitrary, we can take $c_0<\epsilon$. It follows that 
\[ \epsilon > c_0t^{*2}_n\delta^*_n>c_0\delta^*_n\left[\frac{1}{2}-(1-t^*_n)^2\right],
\]
and the above equation yields 
$$\frac{1}{t^{*2}_n}\left[\frac{1}{2}-(1-t^*_n)^2\right] > 1,$$
a contradiction, and the second claim is proved.
\end{enumerate}
\end{proof}

\subsection{Proof of Corollary 1}
\begin{proof}
Recall the definition of $\lambda_n$. For the first part, it suffices to notice that
\begin{align*}
\lim_n \lambda_n(1)&=\lim_n \sup_{0\leq t\leq 1}-\frac{\nu\log(t\sqrt{2\log n}\sigma_n^{-1})}{\log n}-(1-t)^2 + 1 -\beta\\
&=\lim_n \sup_{0\leq t\leq 1}-\frac{\nu\log(\sigma_n^{-1})}{\log n}-(1-t)^2+ 1 -\beta\\
&=\nu \rho + 1 -\beta.
\end{align*}
Therefore $\beta^*(\rho)=\nu \rho+1.$

For the second part, note that
$$\lim_n \lambda_n(1) = \sup_{0\leq t\leq 1}-\frac{2at}{r}-(1-t)^2=\sup_{0\leq t\leq 1}-[t-(1-\frac{a}{r})]^2+(1-\frac{a}{r})^2-\beta.$$
Since $0<1-\frac{a}{r}<1$, it follows that 
$$\lim_n \lambda_n(1)=(1-\frac{a}{r})^2-\beta>0 \quad\iff\quad r>\frac{a}{(1-\sqrt{\beta})}.$$
Therefore $$\beta^*(r) = (1-\frac{a}{r})^2.$$

For the third part, by Lemma \ref{lemma:tail}, 
$$\lim_n \lambda_n(1)
=\lim_{n\to\infty}\sup_{0\leq t\leq 1} \frac{\log[1-\Phi(t\sqrt{2\log n}r^{-1}-\mu)]}{2\log n}-(1-t)^2 + 1 -\beta.$$
Using properties of Gaussian tail probability, it can be easily verified that 
$$\lim_{n\to\infty} \frac{\log\left(1-\Phi(t\sqrt{2\log n}r^{-1}-\mu)\right)}{2\log n}=-\frac{t^2}{\sigma^2r^2}\quad\text{uniformly on $t\in[0,1]$}.$$
Therefore
$$\lim_{n\to\infty}\lambda_n(1)>0\quad\iff\quad\sup_{0\leq t\leq 1} -\left[\frac{t^2}{r^2}+(1-t)^2\right]>1-\beta,$$
and $\beta^*(r) = \frac{r^2}{r^2 + 1}.$
\end{proof}

\subsection{Proof of Theorem 2}

Next, we restate and prove Theorem~\ref{thm:polytail}: 

\polytail*

\begin{proof}
We first improve on Lemma \ref{lemma:tail} and derive a tighter bound on the tail probabilities of the alternative distribution. For any $0<\delta\leq 1$, we have
\begin{equation}
\begin{aligned}
\label{eq:term1and2}
\mathbb{P}_{\mu_n \sim G_n}(X\geq \sqrt{2\delta\log n})& =\int_{-\infty}^{\infty}\left(1-G\left(\frac{\sqrt{2\delta\log n}-z}{\sigma_n}\right)\right)\phi(z)dz.\\
&=\int_{-\infty}^{\sqrt{2\delta\log n}-1}\left(1-G\left(\frac{\sqrt{2\delta\log n}-z}{\sigma_n}\right)\right)\phi(z)dz \\
& \quad+\int_{\sqrt{2\delta\log n}-1}^{\infty}\left(1-G\left(\frac{\sqrt{2\delta\log n}-z}{\sigma_n}\right)\right)\phi(z)dz \\
\end{aligned}
\end{equation}

Because $\sigma_n \to 0$, the tail approximation for $1-G(\mu)$ holds uniformly for $\mu > 1/\sigma_n$. Thus, we can approximate the first term in \eqref{eq:term1and2} as
$$\int_{-\infty}^{\sqrt{2\delta\log n}-1}\left(1-G\left(\frac{\sqrt{2\delta\log n}-z}{\sigma_n}\right)\right)\phi(z)dz\bigg/ \int_{-\infty}^{\sqrt{2\delta\log n}-1} C\left(\frac{\sigma_n}{\sqrt{2\delta\log n}-z}\right)^{\nu}\phi(z)dz \to 1,$$
as $n\to\infty$. It is also straightforward to show that, as $n \to \infty$,
\begin{align*}
\int_{-\infty}^{-(2\delta\log n)^{1/4}} \left(\frac{\sqrt{2\delta\log n}}{\sqrt{2\delta\log n}-z}\right)^{\nu}\phi(z)dz &\;\leq\; \left(\frac{\sqrt{2\delta\log n}}{\sqrt{2\delta\log n}+(2\delta\log n)^{1/4}}\right)^{\nu}\Phi(-(2\delta\log n)^{1/4})\;\to\; 0,\\[15pt]
\int_{-(2\delta\log n)^{1/4}}^{(2\delta\log n)^{1/4}} \left(\frac{\sqrt{2\delta\log n}}{\sqrt{2\delta\log n}-z}\right)^{\nu}\phi(z)dz &\;\to\; 1, \qquad \text{and}\\[15pt]
\int_{(2\delta\log n)^{1/4}}^{\sqrt{2\delta\log n}-1} \left(\frac{\sqrt{2\delta\log n}}{\sqrt{2\delta\log n}-z}\right)^{\nu}\phi(z)dz &\;\leq\; \left(\sqrt{2\delta\log n}\right)^{\nu}\left(1-\Phi((2\delta\log n)^{1/4})\right)\;\to\; 0.
\end{align*}
As a result, we have
\begin{equation}
\label{eq:term1}
    \int_{-\infty}^{\sqrt{2\delta\log n}-1} C\left(\frac{\sigma_n}{\sqrt{2\delta\log n}-z}\right)^{\nu}\phi(z)dz \bigg/ \left(C\frac{(r/\sqrt{\delta})^{\nu}}{n^{1-\beta}}\right) \to 1.
\end{equation}

Let $\epsilon_0=\min\{\beta/2-1/4, 1/2-\beta/2\}$. Turning to the second term in \eqref{eq:term1and2}, we have
\begin{equation}
\label{eq:term2}
    0\leq \int_{\sqrt{2\delta\log n}-1}^{\infty} \left(1-G\left(\frac{\sqrt{2\delta\log n}-z}{\sigma_n}\right)\right)\phi(z)dz \leq 1 - \Phi\left(\sqrt{2\delta\log n}-1\right)
    \leq \frac{1}{n^{\delta-\epsilon_0}}.
\end{equation}
Combining \eqref{eq:term1and2}--\eqref{eq:term2} and recalling the definition of $\sigma_n$, we have
\begin{equation}
\begin{aligned}
\label{eq:polytailtail}
(1+o(1))C\frac{(r/\sqrt{\delta})^{\nu}}{n^{1-\beta}}&\;\leq\;
\mathbb{P}_{\mu_n \sim G_n}(X\geq \sqrt{2\delta\log n})\\
&\;\leq\; (1+o(1))C\frac{(r/\sqrt{\delta})^{\nu}}{n^{1-\beta}} +
\frac{1}{n^{\delta-\epsilon_0}}
\end{aligned}
\end{equation}
For $\delta > 1-\beta+\epsilon_0$ we have
\[
\left(\frac{1}{n^{\delta-\epsilon_0}}\right)\bigg/ \left(C\frac{(r/\sqrt{\delta})^{\nu}}{n^{1-\beta}}\right) = O\left(n^{1-\beta-\delta+\epsilon_0}\right)\to 0.
\]
Therefore for $\delta > 1-\beta+\epsilon_0$,
\[
\mathbb{P}_{\mu_n \sim G_n}(X_n\geq \sqrt{2\delta\log n})\bigg/ \left(C\frac{(r/\sqrt{\delta})^{\nu}}{n^{1-\beta}}\right) \to 1.
\]
Similarly, 
\[
\mathbb{P}_{\mu_n \sim G_n}(X_n\leq - \sqrt{2\delta\log n})\bigg/ \left(C\frac{(r/\sqrt{\delta})^{\nu}}{n^{1-\beta}}\right) \to 1.
\]
Therefore for $\delta > 1-\beta+\epsilon_0$,
\begin{equation}
\label{eq:poly_tail_bound}
\mathbb{P}_{\mu_n \sim G_n}(|X_n|\geq \sqrt{2\delta\log n})\bigg/ \left(2C\frac{(r/\sqrt{\delta})^{\nu}}{n^{1-\beta}}\right) \to 1.
\end{equation}
Suppose that the $1-\alpha$ quantile of $\max_i|X_i|$ under the null is $m(n, \alpha)$. Then the level-$\alpha$ max test rejects the null when $\max_i|X_i|>m(n, \alpha)$. Since $m(n, \alpha)/\sqrt{2\log n} \to 1$, we have
\[
\mathbb{P}_{\mu_n \sim G_n}(|X_n|\geq m(n, \alpha))\bigg/ \left(2C\frac{(r/\sqrt{\delta})^{\nu}}{n^{1-\beta}}\right) \to 1
\]
and
\[
n^{1-\beta}\mathbb{P}_{\mu_n \sim G_n}(|X_n|\geq m(n, \alpha)) \to 2Cr^{\nu}.
\]
Hence the level-$\alpha$ max test satisfies
\begin{equation}
\begin{aligned}
\mathbb{P}_{H_1}(\text{reject }H_0)& =
1-\left(1-(1-n^{-\beta})\mathbb{P}(|N(0, 1)|\geq m(n, \alpha))-n^{-\beta}\mathbb{P}_{\mu_n \sim G_n}(|X_n|\geq m(n, \alpha))\right)^n\\
& =
1-\left(1-(1-n^{-\beta})(1- (1-\alpha)^{1/n})-n^{-\beta}\mathbb{P}_{\mu_n \sim G_n}(|X_n|\geq m(n, \alpha))\right)^n\\
&\to
1-e^{-2Cr^{\nu} + \log(1-\alpha)}, \quad \text{  as $n\to\infty$,}
\end{aligned}
\end{equation}
and the first part of the proposition is proved.

Next we show that modified higher criticism is asymptotically powerless. For modified higher criticism, the critical value of the test $b(n,\alpha)\sim \sqrt{2\log\log n}$. Let $p_i=\PP(|N(0, 1)|\geq |X_i|), i=1,\dots,n$ be the $p$-values. Suppose that under $H_1$, the $p$-values are $i.i.d$ with distribution function $F_n$. Let
\[
\widehat{F_n}(t)=\frac{1}{n}\sum_{i=1}^{n}1_{(p_i\leq t)},
\]
be the empirical distribution of $\{p_i\}, i=1,\dots,n$. 
Let $\tilde{p}_{i} = F_n(p_i)$, and
\[
\widetilde{F_n}(t)=\widehat{F_n}(F_n^{-1}(t))=\frac{1}{n}\sum_{i=1}^{n}1_{(F_n(p_i)\leq t)} = \frac{1}{n}\sum_{i=1}^{n}1_{(\tilde{p}_{i}\leq t)},
\]
Then $\tilde{p}_{i}\simiid \text{Unif}[0, 1]$, and $\{\widetilde{F_n}(t), 0\leq t\leq 1)\}$ follows the same distribution as the empirical distribution of $\{p_i\}, i=1,\dots,n$ under the null. Note that the higher criticism statistics can be decomposed as
\begin{align*}
&\sup_{1/n<t<1/2}\frac{\sqrt{n}(\widehat{F_n}(t)-t)}{\sqrt{t(1-t)}}\\
=&\sup_{F_n(1/n)<t<F_n(1/2)}\frac{\sqrt{n}(\widehat{F_n}(F_n^{-1}(t))-F_n^{-1}(t))}{\sqrt{F_n^{-1}(t)(1-F_n^{-1}(t))}}\\
=&\sup_{F_n(1/n)<t<F_n(1/2)}\frac{\sqrt{n}(\widetilde{F_n}(t)-F^{-1}(t))}{\sqrt{F_n^{-1}(t)(1-F_n^{-1}(t))}}\\
=&\sup_{F_n(1/n)<t<F_n(1/2)}\left(\sqrt{\frac{t(1-t)}{F_n^{-1}(t)(1-F_n^{-1}(t))}}\frac{\sqrt{n}(\widetilde{F_n}(t)-t)}{\sqrt{t(1-t)}}+\frac{\sqrt{n}(t-F_n^{-1}(t))}{\sqrt{F^{-1}(t)(1-F_n^{-1}(t))}}\right).
\end{align*}
We denote
$$A_n(t):=\sqrt{\frac{t(1-t)}{F_n^{-1}(t)(1-F_n^{-1}(t))}},$$ 
$$B_n(t):=\frac{\sqrt{n}(t-F_n^{-1}(t))}{\sqrt{F_n^{-1}(t)(1-F_n^{-1}(t))}}$$
and $$W_n(t):=\frac{\sqrt{n}(\widetilde{F_n}(t)-t)}{\sqrt{t(1-t)}}.$$ 
Note that by Taylor expansion,
\[
A_n(F_n(t))-1\leq \frac{1}{2}\frac{F_n(t)-t}{t}, \quad \text{ for any }t>0.
\]
Let $D(t) = F_n(t)-t = n^{-\beta}\left(\mathbb{P}_{\mu_n \sim G_n}(|X_n|\geq \Phi^{-1}(1-t/2))-t\right)$ and $q_n=(\log n)^3/2n$. Let $\delta_0 = 1-\epsilon$ for $\epsilon>0$ small enough. Then for large enough $n$, by Equation~\ref{eq:poly_tail_bound}
\[
n\sup_{1/n\leq t\leq q_n}D(t)\leq  n^{1-\beta}\left(\mathbb{P}_{\mu_n \sim G_n}\left(|X|\geq \sqrt{2\delta_0\log n}\right)\right)\to 2Cr^{\nu}\delta_0^{-\nu/2}\leq 4Cr^{\nu}.
\]
For large enough $n$, we have
\[
\sup_{F_n(1/n)<t<F_n(q_n)}A_n(t) \leq 1+\frac{n}{2}\sup_{1/n\leq t\leq q_n}D(t)\leq 2Cr^{\nu}
\]
and
$$\sup_{F_n(1/n)<t<F_n(q_n)}B_n(t) \leq n\sup_{1/n\leq t\leq q_n}D(t)\leq 4Cr^{\nu}.$$
Note that $1/n\leq F_n(1/n)$ and $F_n(q_n)\leq q_n+D(q_n)\leq q_n + 4Cr^\nu/n\leq q_n + (\log n)^3/2n = (\log n)^3/n$. Lemma 3 and 4 in \citet{jaeschke1979asymptotic} implies that
$$\sup_{F_n(1/n)<t<F_n(q_n)}W_n(t)/\sqrt{2\log\log n}\leq \sup_{1/n<t<(\log n)^3/n}W_n(t)/\sqrt{2\log\log n}\stackrel{p}{\to} 0.$$
Therefore
$$\mathbb{P}\left(\sup_{F_n(1/n)<t<F_n(q_n)}A_n(t)W_n(t)+B_n(t) > b(n,\alpha)\right) = 0.$$

Write $t=2(1-\Phi(\sqrt{2\delta\log n}))$ for $0<\delta<1$. Then $t\sim n^{-\delta}$ up to $\log n$ factors. Recall that
$\epsilon_0=\min\{\beta/2-1/4, 1/2-\beta/2\}$.
It can be easily verified from Equation \ref{eq:polytailtail} that
$$F_n(t)-t\leq \left\{
\begin{array}{cc}
(1+o(1))C(\frac{r}{\sqrt{\delta}})^{\nu}n^{-1} & \mbox{for $1-\beta+\epsilon_0\leq \delta<1$,}  \\
n^{-(\beta+\delta-\epsilon_0)}& \mbox{for $\epsilon_0<\delta<1-\beta+\epsilon_0$,} \\
n^{-\beta}& \mbox{for $\delta<\epsilon_0$,}
\end{array}\right.$$
Let $q_n^* = 2(1-\Phi(\sqrt{2(1-\beta+\epsilon_0)\log n}))$.
Then $t\geq n^{-\delta-\epsilon_0}$, and it follows that for some constant $C_0$, 
\begin{align*}
&\sup_{F(q_n)<t<F(1/2)}\left(A_n(t)-1\right)\\
\leq &\max\left\{\sup_{q_n<t<q_n^*}A_n(F_n(t))-1, \sup_{q_n^*<t<1/2}A_n(F_n(t))-1\right\}\\
\leq &\max\left\{\frac{C_0}{(\log n)^{3}}, n^{-\beta+2\epsilon_0}\right\} = O\left((\log n)^{-3}\right).
\end{align*}
Similarly we have
\begin{align*}
&\sup_{F_n(q_n)<t<F_n(1/2)}B_n(t)\\
\leq &\max\left\{\sup_{q_n<t<q_n^*}B_n(F_n(t)), \sup_{q_n^*<t<1/2}B_n(F_n(t))\right\}\\
\leq &\max\left\{\frac{C_0}{(\log n)^{3/2}}, n^{-\frac{1}{2}\beta+\frac{1}{4}}\right\} = O\left((\log n)^{-3/2}\right).
\end{align*}
Therefore by Theorem 1 in \citet{jaeschke1979asymptotic}, for large enough $n$ we have
\begin{align*}
&\mathbb{P}\left(\sup_{F_n(q_n)<t<F_n(1/2)}A_n(t)W_n(t)+B_n(t) > b(n,\alpha)\right) \\
\leq &\mathbb{P}\left(\sup_{F_n(q_n)<t<F_n(1/2)}W_n(t) > b(n,\alpha)-\frac{C_0(b(n,\alpha)+1)}{(\log n)^{3/2}}\right)\\
\leq &\mathbb{P}\left(\sup_{0<t<F_n(1/2)}W_n(t) > b(n,\alpha)-\frac{1}{\log n}\right)\to\alpha,
\end{align*}
and the proof is complete.

\end{proof}

\bibliographystyle{apalike}
\bibliography{globaltesting.bib}

\appendix
\appendixpage

\section{Counterexample showing that the condition in Theorem \ref{thm:main} is almost necessary}
Suppose that $\sigma_n=r/\sqrt{2\log n}, r>0$ and $G(\theta)$ is the distribution with $\mathbb{P}(\theta = 3^m)=e^{-3^m}, m=1,2,\dots,$ and $\mathbb{P}(\theta = 0) =1-\sum_{m=1}^{\infty}e^{-3^m}$. Let $\beta=0.52$ and $n_k=e^{5\cdot 3^k},k=1,2,\dots$. Then we have
$$\mathbb{P}\left(\mu_{n_k} = \sqrt{2\cdot (0.2\cdot 3^mr)^2\log n_k}\right)=n_k^{-(0.52+0.2\cdot 3^m)},m=-k,\dots,-1,0,1,\dots.$$
For $m\geq 1$, the probability is less than $n_{k}^{-1.1}$, and the corresponding signal can not be used for detection. For $m\leq 0$, we have $0.52+0.2\cdot 3^m\leq 0.72<0.75$. Therefore for max test to have full asymptotic power, we need
$$\left(0.2\cdot 3^mr\right)^2>\left(1-\sqrt{1-(0.52+0.2\cdot 3^m)}\right)^2\quad\text{for some integer }m\leq 0\Rightarrow r>2.354.$$
For the higher criticism to have full power \citep{donoho2004higher}, we need
$$(0.2\cdot 3^mr)^2>0.52+0.2\cdot 3^m-0.5\quad\text{for some integer }m\leq 0\Rightarrow r>2.345.$$
Since $2.345<2.354$, the detection boundary for higher criticism is smaller than that of max test despite $F$ being exponential.

\section{Proof of Propositions \ref{prop:HCoptimal} and \ref{prop:lambda-power}(b) for the modified higher criticism and Berk-Jones tests}

Since Proposition \ref{prop:HCoptimal} is a directly corollary of Proposition \ref{prop:lambda-power}, we will only provide the proof of Part (b) of Proposition \ref{prop:lambda-power} for the modified higher criticism and Berk-Jones tests.

\begin{proof}

Under the condition of Part (b), there exists $\delta_0\in (0, 1)$ and constant $c_0>0$ such that $\displaystyle \lim_{n\to\infty}\lambda_n(\delta_0) - \frac{1-\delta_0}{2} = 2c_0>0$ for large enough $n$. Let $t=\bar\Phi(2\delta_0\log_n)<n^{-\delta_0}$. Recall that $F_n$ is the empirical distribution of $p$-values. Therefore $nF_n(t)=N(\delta_0)$ follows a binomial distribution
with 
\[
\EE_{H_1}N(\delta_0) = nt(1-n^{-\beta}) + n^{\lambda_n(\delta_0)}\geq nt(1-n^{-\beta}) + n^{\frac{1-\delta_0}{2} + c_0}
\geq nt + \frac{1}{2}n^{\frac{1-\delta_0}{2} + c_0}.
\] 
for large enough $n$, and
\[
\text{Var}_{H_1}N(\delta_0) = \EE_{H_1}N(\delta_0)\left(1-\frac{\EE_{H_1}N(\delta_0)}{n}\right)
\]
Therefore, by Chebyshev's inequality,
\begin{align*}
\PP_{H_1}[N(\delta_0) < nt + n^{\frac{1-\delta+c_0}{2}}] &\leq \frac{\text{Var}_{H_1}N(\delta_0)}{\left(\EE_{H_1}N(\delta_0) - nt - n^{\frac{1-\delta_0+c_0}{2}}\right)^2}\\
& \leq \frac{\EE_{H_1}N(\delta_0)}{\left(\EE_{H_1}N(\delta_0) - nt - n^{\frac{1-\delta_0+c_0}{2}}\right)^2}\\
& \leq \frac{1}{\EE_{H_1}N(\delta_0)-nt-2n^{\frac{1-\delta_0+c_0}{2}}}\\
& \leq n^{-\frac{1-\delta_0+c_0}{2}}
\end{align*}
for large enough $n$.
Therefore, for the modified higher criticism statistics, we have
\begin{align*}
\PP_{H_1}(mHC_n \geq 2\sqrt{\log\log n}) & \geq \PP_{H_1}\left(\frac{\sqrt{n}(F_n(t)-t)}{\sqrt{t(1-t)}}\geq 2\sqrt{\log\log n}\right) \\
& \geq \PP_{H_1}\left(\frac{N(\delta_0)-nt}{\sqrt{nt(1-t)}}\geq 2\sqrt{\log\log n}\right) \\
& \geq 1 - \PP_{H_1}\left[N(\delta_0) < nt + n^{\frac{1-\delta+c_0}{2}}\right] \to 1
\end{align*}
as the $n\to\infty$, where the last inequality holds for large enough $n$.
Now we turn to the Berk-Jones statistics. First, it can be easily verified that $\log(x+1)>x/2$ for $x\in(-1/2, 1/2)$. Without loss of generality, suppose that $2c_0 < (1-\delta_0)/2$, then $\EE_{H_1}N(\delta_0)/nt\to 1$, and $F_n(t)/t \stackrel{p}{\to} 1$. If $1/2<F_n(t)/t<3/2$, then
\begin{align*}
& 2n\left[F_n(t)\log\frac{F_n(t)}{t} + (1-F_n(t))\log\frac{(1-F_n(t))}{(1-t)} \right]\\
\geq \;\ & nF_n(t)\left(\frac{F_n(t)}{t}-1\right) + n(1-F_n(t))\left(\frac{(1-F_n(t))}{(1-t)}-1\right) = \frac{n(F_n(t)-t)^2}{t(1-t)}. 
\end{align*}
Therefore
\begin{align*}
\lim_{n\to\infty}\PP_{H_1}(BJ_n \geq 2\sqrt{\log\log n}) & \geq \lim_{n\to\infty}\PP_{H_1}\left(\frac{n(F_n(t)-t)^2}{t(1-t)}\geq 4\log\log n\right) = 1,
\end{align*}
which completes the proof.
\end{proof}

\section{Additional simulation results}\label{app:additional-simulations}

We provide additional simulation results where $G$ is the Gaussian (Figure~\ref{fig:gaussianpower}), logistic (Figure~\ref{fig:logisticpower}), chi-squared (Figure~\ref{fig:chipower}),
$t_5$ (Figure~\ref{fig:t5power}), and $t_3$ (Figure~\ref{fig:t3power}) distribution, and $G_n=rG$. In each simultion, $n=50,000$ and there are $n_1 = \lfloor n^{1-\beta}\rfloor$ non-null means drawn from $G_n$. We find that in all settings, the power of max test is similar to the power of the higher criticism test when $\beta>1/2$.

\begin{figure}
\centering
\includegraphics[width = .8\textwidth]{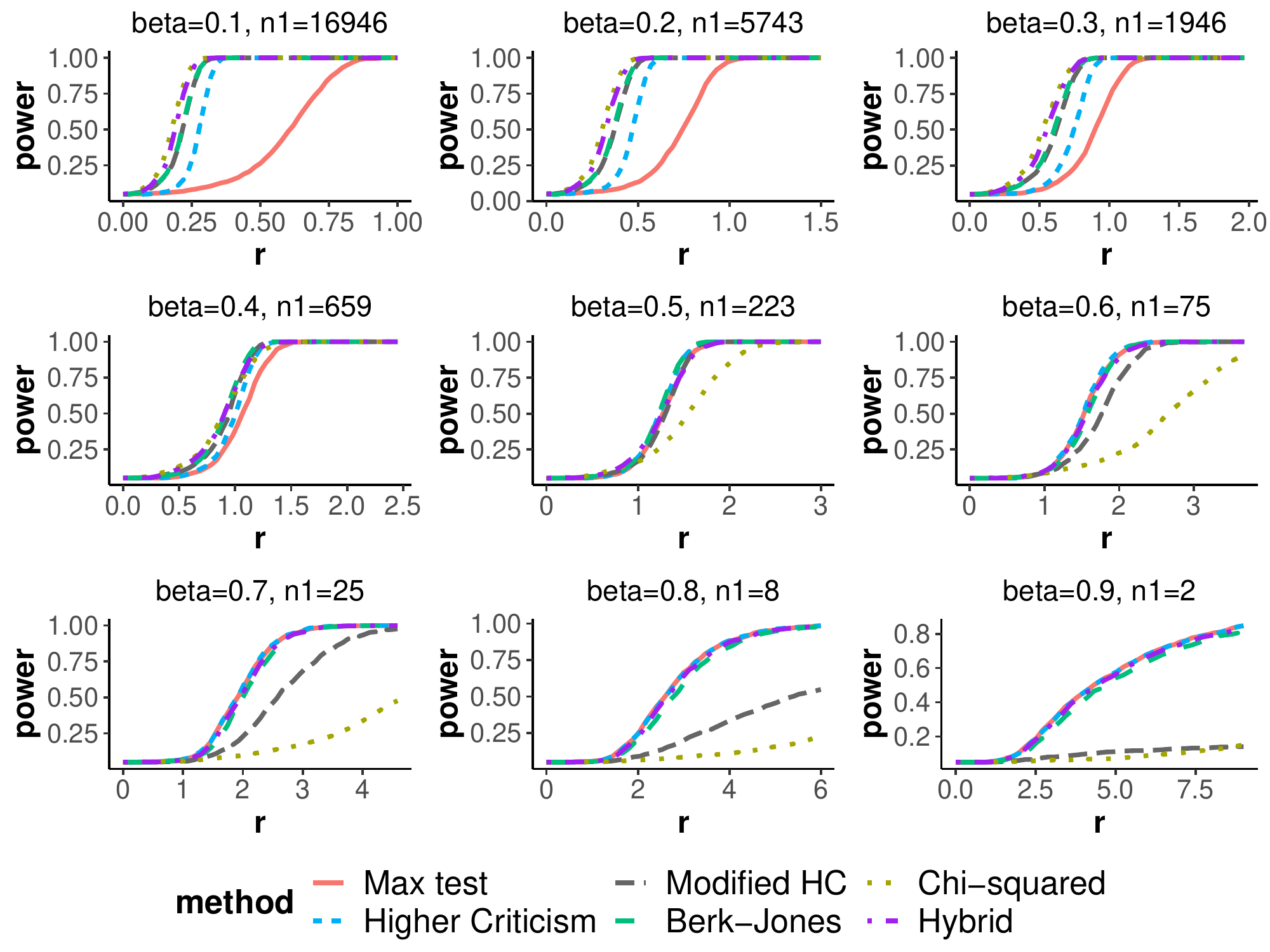}
\caption{$G = N(0, 1)$}
\label{fig:gaussianpower}
\end{figure}

\begin{figure}
\centering
\includegraphics[width = .8\textwidth]{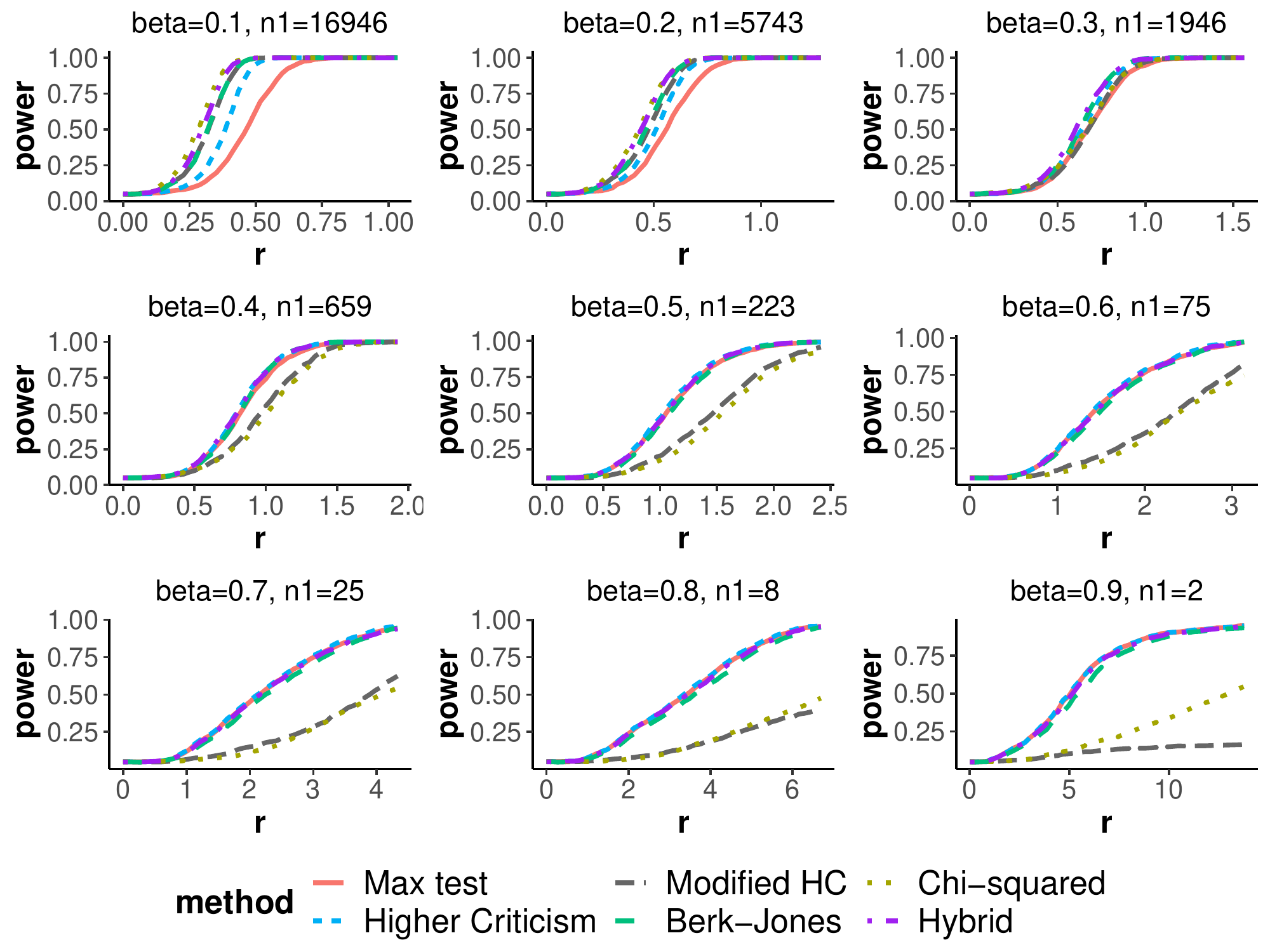}
\caption{$G = \text{Logistic}(0, 1)$}
\label{fig:logisticpower}
\end{figure}

\begin{figure}
\centering
\includegraphics[width = .8\textwidth]{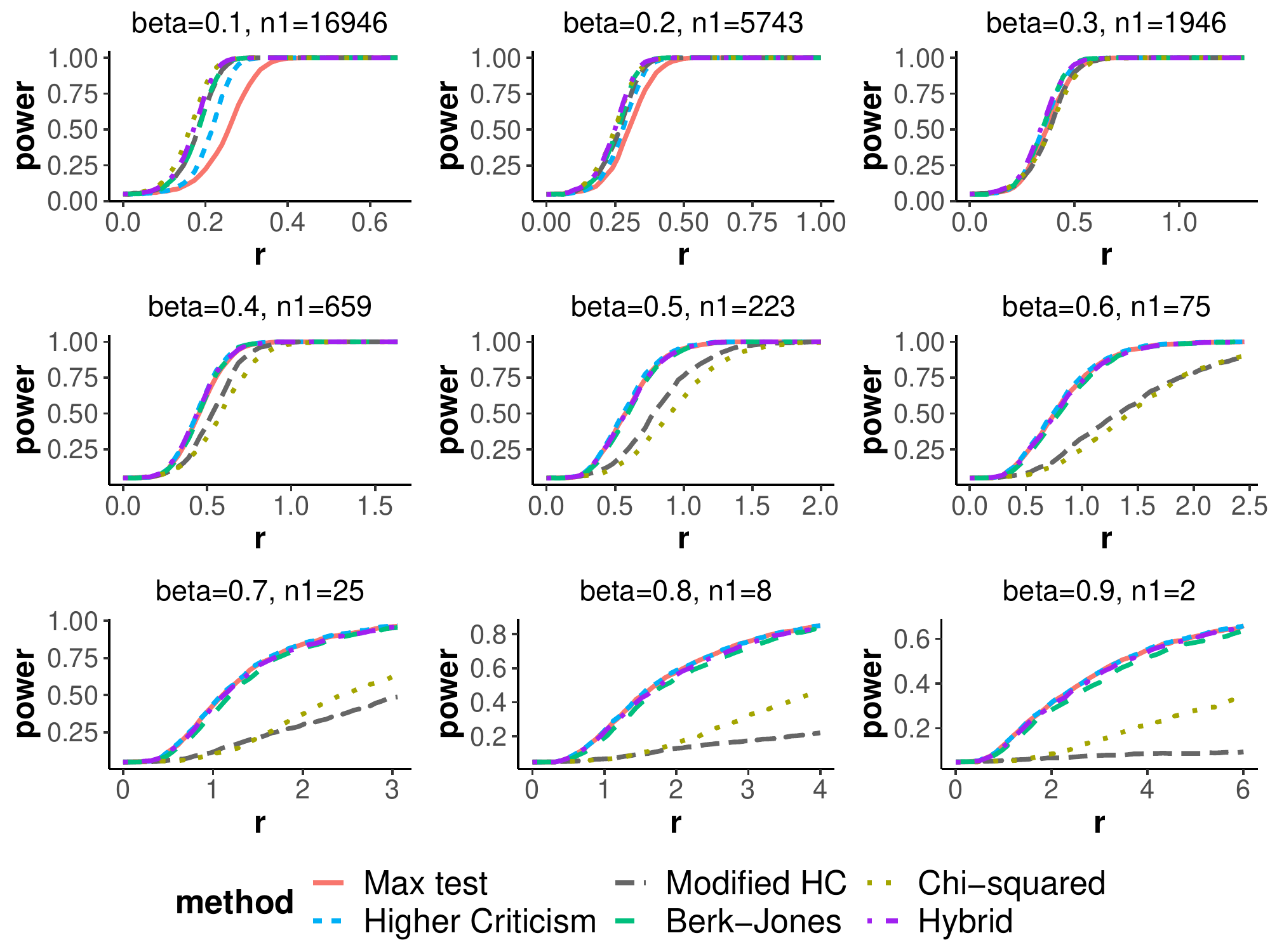}
\caption{$G = \chi^2(1)$}
\label{fig:chipower}
\end{figure}

\begin{figure}
\centering
\includegraphics[width = .8\textwidth]{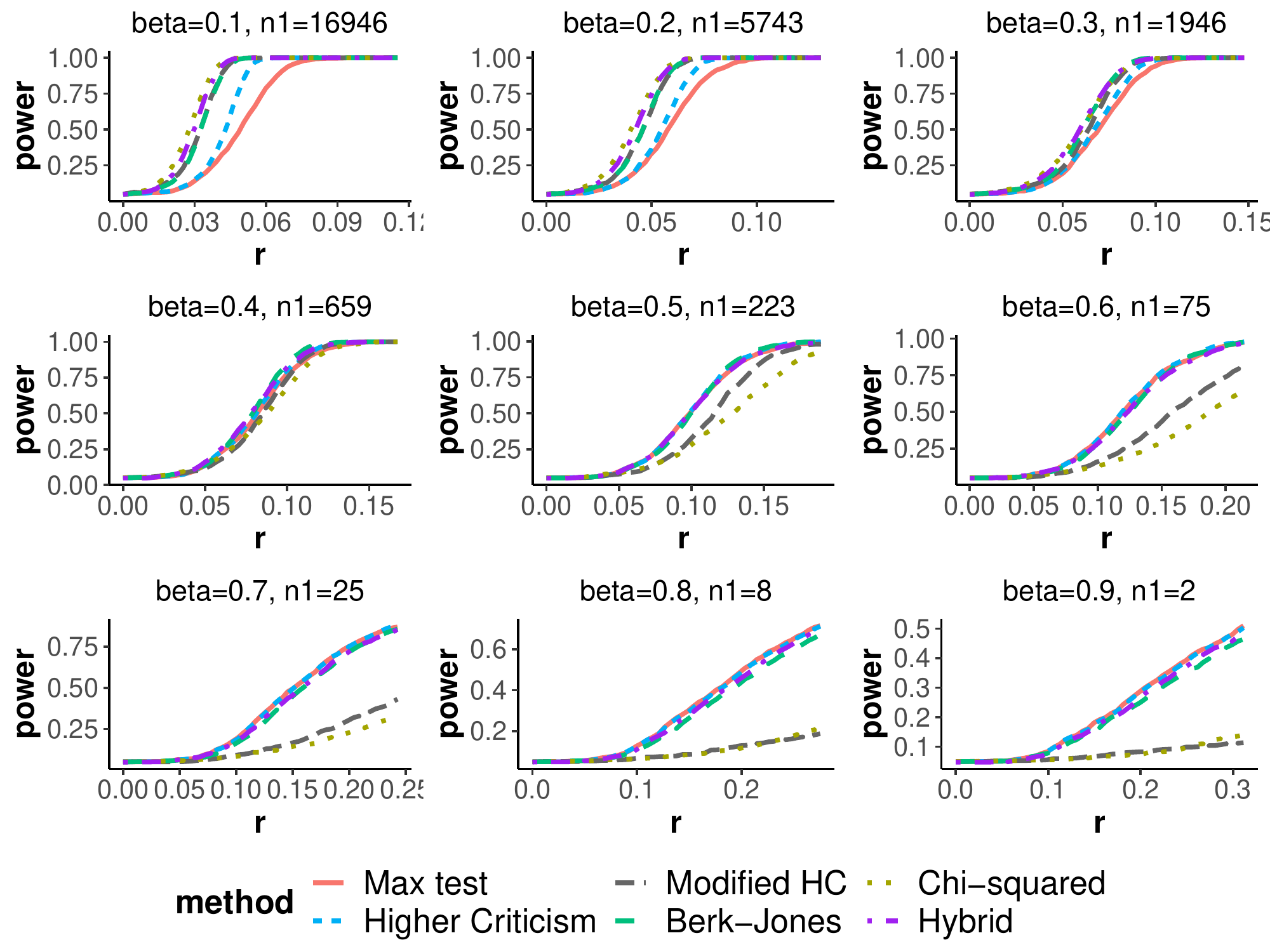}
\caption{$G=t_5$}
\label{fig:t5power}
\end{figure}

\begin{figure}
\centering
\includegraphics[width = .8\textwidth]{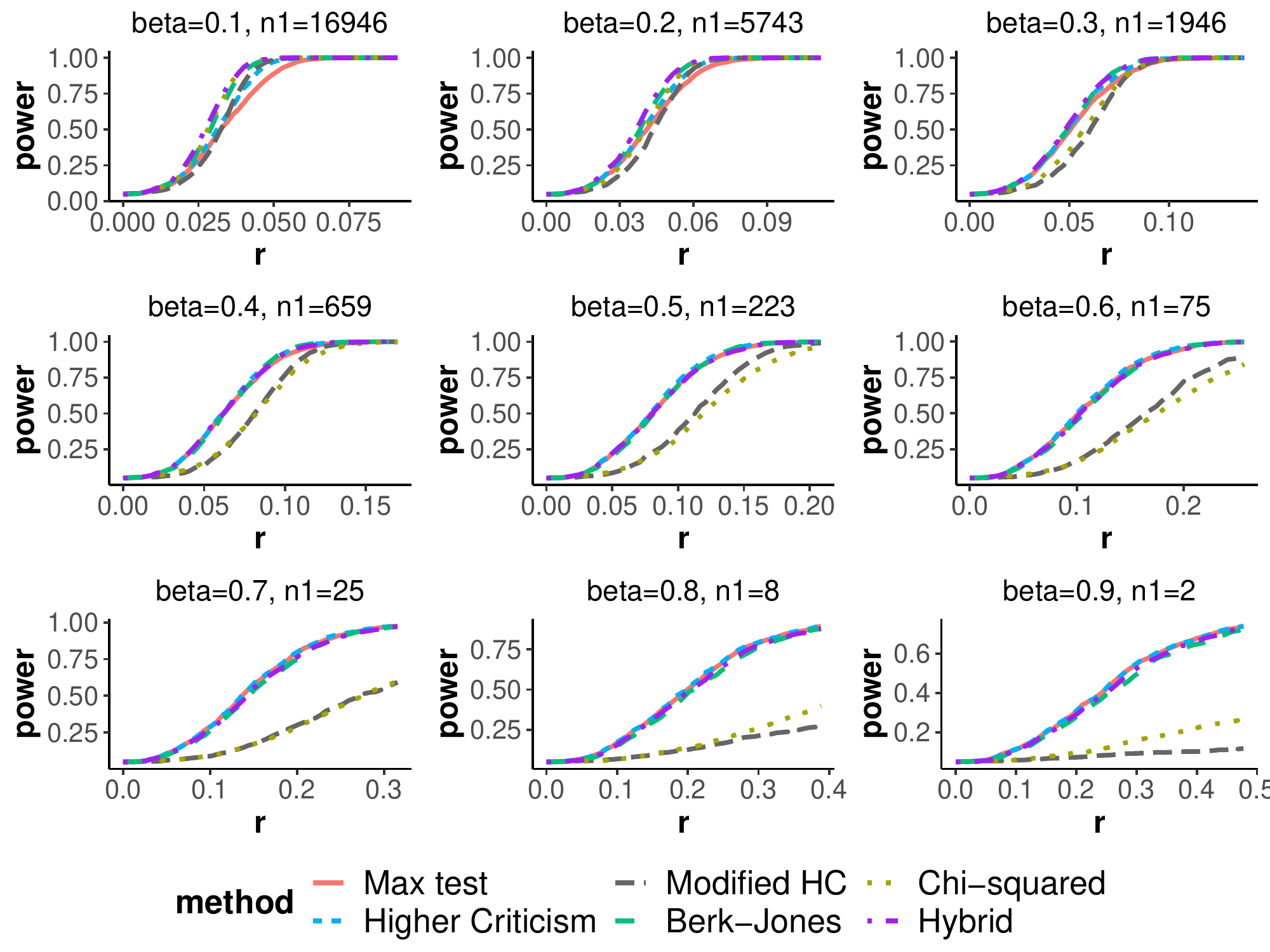}
\caption{$G=t_3$}
\label{fig:t3power}
\end{figure}

\end{document}